\documentclass[11pt,a4paper,twoside]{article}
\usepackage[utf8]{inputenc}
\usepackage{amsmath,amsfonts,amssymb,amsthm,bbm,latexsym,mathrsfs}
\usepackage{graphicx,color,epsfig,fancyhdr,dsfont}
\usepackage{enumerate}
\usepackage{hyperref}
\usepackage{indentfirst}
\usepackage[title,titletoc]{appendix}
\usepackage[all]{hypcap}
\usepackage{placeins}
\usepackage[affil-it]{authblk}
\usepackage{color}
\usepackage[includeheadfoot,margin=3.5cm]{geometry}
\usepackage{subcaption}
\usepackage{xcolor}
\usepackage{float}
\usepackage{comment}  
\allowdisplaybreaks[4]
\usepackage{textcase}
\usepackage{lineno}
\allowdisplaybreaks

\allowdisplaybreaks

\newtheorem{theorem}{Theorem}[section]
\newtheorem{definition}[theorem]{Definition}
\newtheorem{lemma}[theorem]{Lemma}

\newtheorem{remark}[theorem]{Remark}

\newtheorem{condition*}[theorem]{Condition}

\newtheorem{assumption}{Assumption}

\def\be{\mathbb{E}}
\def\br{\mathbb{R}}

\def\bn{\mathbb{N}}

\def\cb{\mathcal{B}}
\def\cd{\mathcal{D}}
\def\cl{\mathcal{L}}
\def\cm{\mathcal{M}}
\def\cp{\mathcal{P}}
\def\cf{\mathcal{F}}

\def\cw{\mathcal{W}}
\def\d{\mathrm{d}}
\def\re{\mathrm{e}}

\def\vp{\varphi}
\def\ve{\varepsilon}

\def\b1{{\mathbbm{1}}}

\numberwithin{equation}{section}

\newcommand{{\X}}{{\mathbb{X}}}                     

\title{Strong averaging principle for multiscale time-inhomogeneous SDEs with multiplicative $\alpha$-stable noises}

\author[a]{Jiaquan Lu}
\author[b,c]{Huaizhong Zhao}
\affil[a]{Research Center for Mathematics and Interdisciplinary Sciences, Shandong University, Qingdao 266237, China}
\affil[b]{Department of Mathematical Sciences, Durham University, DH1 3LE, UK}
\affil[c]{School of Mathematics, Shandong University, Jinan 250100, China}
\affil[ ]{jiaquan.lu@mail.sdu.edu.cn, huaizhong.zhao@durham.ac.uk}
\date{~}

\begin{document}
	\maketitle
    \vskip-40pt
	\vspace{-20pt}
\begin{abstract}
In this paper, we study the strong averaging principle for multiscale time-inhomogeneous stochastic systems driven by multiplicative $\alpha$-stable processes with $\alpha\in(1,2)$. Based on Khasminskii's discretization approach, we first establish that the fast component processes with a frozen slow variable admits a periodic measure. We then prove the strong convergence of the slow subsystem to an averaged system that depends on the time scale $\ve$. For any fixed $\ve$, if the reciprocals of the two periods $\tau_1$ and $\ve\tau_2$ are rationally linearly independent, an important consequence is that the averaged system has random quasi-periodicity. Furthermore, by applying the ergodic theorem, we prove the strong convergence of the slow subsystem to another averaged system, a time-inhomogeneous SDEs independent of the time scale $\ve$. Our result is also novel even in the time-homogeneous case for a fully coupled multiscale system with multiplicative $\alpha$-stable noises. Finally, we apply the result to a climate-weather system.
    \medskip
	
\end{abstract}

\section{Introduction}
In dynamical systems, multiscale phenomena are commonly observed in both the natural and social sciences. Examples include the division of time into days, months and years corresponding directly to the multiple scales of motion observed in the solar system  (E and Bjorn \cite{EB03}), processes occur on timescales ranging from protein-protein interactions within seconds subsequent changes in the cellular environment unfolding over hours (Bertram and Rubin \cite{BR17}), and from volatility fluctuations on asset price dynamics in financial markets to relatively slow-changing status of an economy, especially of a large economy (Fouque, Papanicolaou, and Sirca \cite{FPS00}). These systems are commonly characterised by mathematical models based on coupled stochastic differential equations (SDEs for short) or stochastic partial differential equations (SPDEs for short) with multiple time scales, where distinct equations describe the slow and fast subsystems, respectively. For these multiscale (or slow-fast) SDEs and SPDEs, the averaging principle serves as a powerful analytical tool.

On the other hand, periodicity exists in many real world problems including such slow-fast systems, in which the slow components and fast components may have different periods. For example, in the climate-weather system (Benzi, Parisi, Sutera, and Vulpiani \cite{BPSV83}, Imkeller and Von Storch \cite{IV01}, F. Benth and J. Benth \cite{BB07}), the slow climate system has about a $100,000$-year period and the fast weather system has a 1-year period of seasonal change (and a $1/365.25$-year period of daily weather change if further periodicity is explored). In the past 20 years, it has become apparent that the intertwining of periodicity and randomness yields random periodicity; see, for example, Zhao and Zheng \cite{ZZ09}, Chekroun, Simonnet, and Ghil \cite{CSG11}, Feng and Zhao \cite{FZ20}, Lu and Zhao \cite{LZ26}, to name but a few. Moreover, many real world problems are subject to extreme fluctuations of noise, for instance, extreme weather variations, sharp change of stock prices, or other economic conditions. Thus, the study of multiscale periodic SDEs driven by $\alpha$-stable processes is an interesting and important problem.

The foundations of the averaging principle lie in the classical works in celestial mechanics by Clairaut \cite{C54}, Laplace \cite{L25}, and Lagrange \cite{L53}. Later, Bogoliubov and Mitropolsky \cite{BM61} established a rigorous theoretical framework for the averaging principle in ordinary differential equations. Subsequently, Khasminskii \cite{K68} generalized the averaging principle of stochastic systems and initiated systematic investigations of the averaging principle in this direction. 

Khasminskii considered the following multiscale time-homogeneous SDEs:
\begin{equation*}
    \begin{cases}
        \d X_t^\ve=b(X_t^\ve,Y_t^\ve)\d t+\sigma(X_t^\ve,Y_t^\ve)\d W_t^1,~~X_0^\ve=x\in\br^n,\\
        \d Y_t^\ve=\frac{1}{\ve}f(X_t^\ve, Y_t^\ve)\d t+\frac{1}{\sqrt{\ve}}g(X_t^\ve,Y_t^\ve)\d W_t^2,~~Y_0^\ve=y\in\br^n,
    \end{cases}
\end{equation*}
where $\{W_t^1,t\geq 0\}$ and $\{W_t^2,t\geq0\}$ are two mutually independent Brownian motions in ${\br }^d$, $\ve$ is a small parameter in $(0,1)$ characterising the separation of time scales between the slow process $X_t^\ve$ and the fast process $Y_t^\ve$. The drift and diffusion coefficients satisfy the Lipschitz condition and are uniformly bounded with respect to $y$. It is also assumed that the ergodic condition holds in the sense that there exist two functions $\bar{b}$ and $\bar{a}$ such that for any $x,y\in\br^n$ and $t\geq0, T>0$,
\begin{align*}
    \left|\frac{1}{T}\int_t^{t+T}\be[b(x,Y_r^{x,y})]\d r-\bar{b}(x)\right|\leq \beta(T)(1+|x|^2),
\end{align*}
and
\begin{align*}
    \left|\frac{1}{T}\int_t^{t+T}\be[a(x,Y_r^{x,y})]\d r-\bar{a}(x)\right|\leq \beta(T)(1+|x|^2),
\end{align*}
where $a(x,y)=\sigma^*\sigma(x,y)$, $\beta(T)$ converges to zero as $T$ goes to infinity, and $Y_t^{x,y}$ is a solution to the following frozen equation
\begin{equation*}
    \begin{cases}
        \d Y_t^{x,y}=f(x,Y_t^{x,y})\d t+g(x,Y_t^{x,y})\d W_t^2,\\
        Y_0=y.
    \end{cases}
\end{equation*}
Under the above assumptions, Khasminskii established the averaging principle, according to which the slow component $X_t^\varepsilon$ converges weakly to the corresponding averaged system $\bar{X}_t$, where $\bar{X}_t$ is the unique solution to
\begin{equation*}
    \begin{cases}
        \d \bar{X}_t=\bar{b}(\bar{X}_t)\d t+\bar{a}(\bar{X}_t)\d W_t^1,\\
        \bar{X}_0=x.
    \end{cases}
\end{equation*}
Khasminskii’s proof was mainly based on a time-discretization technique. Since then, it has become a fundamental tool in the handling of multiscale systems. Later, Pardoux and Veretennikov \cite{PV01,PV03} developed a more analytic approach with Poisson equations as an alternative to Khaminskii's discretization method. 

Thereafter, many important results on averaging principles for multiscale SDEs and SPDEs have been obtained, such as Veretennikov \cite{V91}, Cerrai \cite{C09}, Freidlin and Wentzell \cite{FW98}, E, Liu, and Vanden-Eijnden \cite{ELV05}, Dong, Sun, Xiao, and Zhai \cite{DSXZ18}, Wang and Roberts \cite{WR12}, Cerrai and Freidlin \cite{CF09}, Br{\'e}hier \cite{B12}, R\"ockner and Xie \cite{RX21}, Hairer and Li \cite{HL20}, to name but a few. A crucial condition in the literature for obtaining a stronger convergence in probability of the slow process is that the diffusion coefficient is independent of the fast component. Moreover, all the references cited above dealt with time-homogeneous systems driven by Brownian motions, while time-inhomogeneous systems are essential for capturing time-dependent physical systems, such as a damped oscillator under periodic forcing. For a multiscale time-homogeneous system, the frozen equation associated with the fast subsystem is a time-homogeneous SDEs or SPDEs. Under some mild conditions, this equation admits an unique invariant measure. 

In contrast to the time-homogeneous case, time-inhomogeneous systems generally do not admit invariant measures in the usual sense. To address this gap in the study of multiscale time-inhomogeneous systems, Liu, R\"ockner, Sun, and Xie \cite{LRSX20} fixed the time variable in the frozen equation to obtain an invariant measure. Related applications of this method can be found in de Feo \cite{D21} and R\"ockner, Sun, and Xie \cite{RSX21}. Another approach to compensating for the lack of an invariant measure in the time-inhomogeneous frozen equations is to introduce the evolution system of measures. This approach was discussed, for example, in Cerrai and Lunardi \cite{CL17}, Wainrib \cite{W13}, Uda \cite{U21}, Cheng, Sun, and Xie \cite{CSX25}.

The aforementioned multiscale stochastic dynamical systems are all perturbed by continuous noise, whereas discontinuous noise is a key feature in the modelling of many physical systems. A typical example in finance is that a release of major news can give rise to extreme fluctuations in stock prices, resulting in a return distribution characterised by high kurtosis and heavy tails. To better model these systems, many researchers have focused on multiscale SDEs and SPDEs with jump-type noise, particularly by $\alpha$-stable noise. Here, we briefly review some related works on this topic. Sun, Xie, and Xie \cite{SXX22} considered the following multiscale stochastic dynamical systems driven by additive $\alpha$-stable noise
\begin{equation*}
    \begin{cases}
        \d X_t^\ve=b(X_t^\ve,Y_t^\ve)\d t+\d L_t^1,~~X_0^\ve=x\in\br^n,\\
        \d Y_t^\ve=\frac{1}{\ve}f(X_t^\ve,Y_t^\ve)\d t+\frac{1}{\ve^{1/\alpha}}\d L_t^2,~~Y_0^\ve=y\in\br^m,
    \end{cases}
\end{equation*}
they used the Poisson equation technique to derive the optimal rates of both strong and weak convergence. Building on the finite dimensional framework of \cite{SXX22}, Sun and Xie \cite{SX23} extended the analysis to the infinite dimensional case. Bao, Yin and Yuan \cite{BYY17} used Khasminskii's discretization approach to the averaging principle for SPDEs with two-time-scale Markov switching.  Zhang, Huang, Wang, Wang, and Duan \cite{ZHWWD24} obtained the weak averaging principle using an analysis of a nonlocal Poisson equation. Recently, Li, Sun, Wang, and Xie \cite{LSWX25} weakened the assumption of a bounded drift coefficient required in \cite{SXX22} and applied Khasminskii's discretization approach to show that a strong averaging principle holds for time-inhomogeneous SPDEs. Note that all the aforementioned SDEs and SPDEs are perturbed by additive noise.
 
In this paper, we focus on the multiscale time-inhomogeneous dynamical systems driven by multiplicative $\alpha$-stable noise
\begin{equation}\label{SDEs}
\begin{cases}
 \d X_t^{\ve}=b(t,X_t^{\ve},Y_t^{\ve})\d t+\sigma(t,X_t^{\ve})\d L_t^1,~~X_s^\ve=x\in\br^n,\\
 \d Y_t^{\ve}=\frac{1}{\ve}f(t/\ve,X_t^{\ve},Y_t^{\ve})\d t+\frac{1}{\ve^{1/\alpha}}g(t/\ve,Y_t^{\ve}) \d L_t^2,~~Y_s^\ve=y\in\br^m,
\end{cases}
\end{equation}
where $\{L_t^1,t\in\br\}$ and $\{L_t^2,t\in\br\}$ are mutually independent rotationally invariant $\alpha$-stable processes with dimensions $d_1$ and $d_2$ and stability index $\alpha\in(1,2)$, respectively, and their associated L\'evy measures are denoted by $\nu_1$ and $\nu_2$. It is worth noting that the coefficients of the slow component in \eqref{SDEs} are allowed to depend explicitly on time $t$ so as to capture the time-inhomogeneous nature of the slow subsystem. In contrast, the coefficients of the fast equation depend on $t/\varepsilon$ rather than $t$, which is a key ingredient in the derivation of the averaging principle. To see this, we freeze the slow variable at $X_t^\ve=x$, introduce the rescaled time variable $u=t/\ve$, and define the rescaled fast process $\tilde{Y}_u=Y_{\ve u}^\ve$, then the fast subsystem $Y_t^\ve$ can  be rewritten in the form 
\[\d \tilde{Y}_u=f(u,x,\tilde{Y}_u)\d u+g(u,\tilde{Y}_u)\d \tilde{L}_u^2,\]
where $L_{\ve u}^2\stackrel{d}{=}\ve^{1/\alpha}\tilde{L}_u^2$. Thus SDEs \eqref{SDEs} is a natural model for physical systems with multiple components, in particular, the $t/\ve$ appears exactly in the equation of fast component in the time-inhomogeneous case.  This allows us to establish the ergodicity of the fast nonstationary component. Moreover, see Section \ref{sec:Example} and Appendix A for motivating examples.

We only consider the case where these two stable processes have the same stability index for simplicity. When they have different indexes, our results are still valid. As far as we are aware, results in the literature on multiscale stochastic systems with multiplicative $\alpha$-stable noise remain quite limited. So far only Chen, Hao, and Zhang \cite{CHZ25} have investigated the averaging principle for SDEs using PDE methods. Their framework applies only to such slow–fast systems, in which the slow component evolved according to an SDE with multiplicative $\alpha$-stable noise, whereas the fast process was not described by an SDE but instead assumed to satisfy a strong law of large numbers and independent of the slow variable.

In contrast to the existing results, the main contribution of this work is to extend the results for multiscale stochastic systems with additive noise to multiplicative noise, using Khasminskii's time discretization method. To the best of our knowledge, this is the first work in establishing a strong averaging principle for multiscale time-inhomogeneous SDEs driven by multiplicative $\alpha$-stable noise, which is also novel even in the time-homogeneous case for a fully coupled slow-fast system with multiplicative $\alpha$-stable noises. It is worth noting that the dependence of the solutions on the initial conditions is key in making progress in this analysis. In the context of time-inhomogeneous system, we first prove that the fast component processes with frozen slow variable possess a periodic measure under some mild condition. We then apply the ergodic theory and the law of large numbers for periodic measures developed by Feng and Zhao \cite{FZ20} to derive the averaged system for the slow component. For any fixed $\ve$, an important consequence is that when both the slow and fast components possess different periodicities, the averaged system with $\ve$ has a quasi-periodicity when the reciprocals of the two periods $\tau_1$ and $\ve\tau_2$ are rationally linearly independent. This is a new observation in the literature. Random quasi-periodicity was studied recently by Feng, Qu and Zhao \cite{FQZ21}, Liu and Lu \cite{LL25}. By further applying Birkhoff's ergodic theorem of the dynamical system, we obtain a time-inhomogeneous averaged system that serves as an appropriate slow component. Our result applies to the climate-weather slow-fast system in Section \ref{sec:Example} and we obtain a quasi-periodic averaged climate-weather system, see \eqref{eq:w-c} for details, where we can take $\ve$ approximately to be $10^{-5}$ in realty. 

The remaining parts of the paper are structured as follows. Section \ref{sec:Assumptions and main results} elaborates on some common assumptions and presents the main results. The rigorous proofs of these main results are developed in Section \ref{sec:Proof of the main results}. More precisely, Sections \ref{subsec:Some prior estimates} and \ref{subsec:The auxiliary process hat{Y}_t^ve for the fast system} are devoted to deriving some a priori estimates for the coupled system and introducing the auxiliary process corresponding to the fast subsystem. In Section \ref{subsec:Ergodicity for the fast system with frozen slow variable}, we establish the ergodicity of the frozen equation. We further prove the two strong convergence properties of the slow subsystem to the averaged equations in Sections \ref{subsec:The averaged equation} and \ref{subsec:The averaged equation without ve}. Section \ref{sec:Example} provides an illustrative example to demonstrate the application of the theoretical results. Finally, Appendix \ref{A} provides a detailed description of the climate and weather systems, which serves as a motivating example for our framework.
\newpage

\section{Assumptions and main results}\label{sec:Assumptions and main results}
We start by recalling some key concepts and notations related to $\alpha$-stable processes that are essential for the subsequent analysis. Let $\{L_t,t\geq0\}$ be an $\mathbb{R}^d$-valued rotationally invariant $\alpha$-stable process, where the stability index satisfies $\alpha\in(1,2)$, whose infinitesimal generator $\cl$ is given for $\vp\in C_b^2(\br^d)$ by
\[\cl \vp(x)=\int_{\br^d\setminus\{0\}}[\vp(x+z)-\vp(x)-\b1_{\{0<|z|\leq 1\}}(z)\langle \nabla\vp(x),z\rangle]\nu(\d z),\]
where $\nu(\mathrm{d}z)$ denotes the associated L\'evy measure, which is characterized by
\[\nu(\d z)=\frac{C_{\alpha,d}}{|z|^{d+\alpha}}\d z\]
with $C_{\alpha,d}=\alpha 2^{\alpha-1}\pi^{-\frac{d}{2}}\Gamma\left(\frac{d+\alpha}{2}\right)/\Gamma\left(1-\frac{\alpha}{2}\right)$. Then for any $1\leq p<\alpha$, we have
\[\int_{|z|>1}|z|^p\nu(\d z)=\frac{C_{\alpha,d}S_d}{\alpha-p}\]
and
\[\int_{|z|\leq 1}|z|^2\nu(\d z)=\frac{C_{\alpha,d}S_d}{2-\alpha},\]
where $S_d$ is the surface area of the unit sphere $\mathbb{S}^{d-1}\subset\mathbb{R}^d$, which is given by $S_d=2\pi^{\frac{d}{2}}/\Gamma\left(\frac{d}{2}\right)$. Let $N(\d z,\d t)$ denote the corresponding Poisson random measure, satisfying
\[N((s,t]\times\Gamma)=\sum_{r\in(s,t]}\b1_{\Gamma}(L_r-L_{r-}),\]
for any $s\leq t$ and $\Gamma\in\cb(\br^d\setminus\{0\})$. According to the L\'evy-It\^o decomposition theorem \cite{A09,PZ07}, the process admits the representation
\[\d L_t=\int_{|z|\leq1}z\tilde{N}(\d z,\d t)+\int_{|z|>1}zN(\d z,\d t),\]
where $\tilde{N}(\d z,\d t):=N(\d z,\d t)-\nu(\d z)\d t$ is the compensated Poisson random measure. 

In our framework, the time index  is defined over the entire real line $\br$, thus it is necessary to extend the one-sided (for $t\geq 0$) rotationally invariant $\alpha$-stable process to the whole time domain. To this end, let $\{Z_t^1,t\geq0\}$ be an $\mathbb{R}^d$-valued symmetric and rotationally invariant $\alpha$-stable process, and let $\{Z_t^2,t\geq0\}$ be an independent copy of  $\{Z_t^1,t\geq0\}$. Then the corresponding two-sided process $\{L_t,t\in\br\}$ is constructed as
\begin{equation*}
   L_t= \begin{cases}
        Z_t^1, &t\geq0,\\
        Z_{-t}^2,&t<0.
    \end{cases}
\end{equation*}

We now introduce the assumptions on the coefficients and fix the notations used throughout the paper. Let $\langle\cdot,\cdot\rangle$ denote the Euclidean inner product, $|\cdot|$ the Euclidean norm, and $\Vert\cdot\Vert$ the standard Frobenius norm. 

\begin{assumption}\label{cond}
The drift coefficients $b:\br\times\br^n\times\br^m\to\br^n$ and $f:\br\times\br^n\times\br^m\to\br^m$, and diffusion coefficients $\sigma:\br\times\br^n\to\br^{n\times d_1}$ and $g:\br\times\br^m\to\br^{m\times d_2}$ are assumed to satisfy
\begin{description}
\item[(A1)] There exist constants $C, C_\sigma,C_g>0$ such that for all $x_i\in\br^n$, $y_i\in\br^m$ and $t_i\in\br$, $i=1,2$,  
\begin{equation}\label{cond:Lip}
    \begin{aligned}
        &|b(t_1,x_1,y_1)-b(t_2,x_2,y_2)|\leq C(|t_1-t_2|^\frac{1}{\alpha}+|x_1-x_2|+|y_1-y_2|),\\
        &|f(t_1,x_1,y_1)-f(t_1,x_2,y_2)|\leq C(|x_1-x_2|+|y_1-y_2|),\\
        &\Vert\sigma(t_1,x_1)-\sigma(t_1,x_2)\Vert+\Vert g(t_1,y_1)-g(t_1,y_2)\Vert\leq C_\sigma|x_1-x_2|+C_g|y_1-y_2|.
    \end{aligned}
\end{equation}
\item[(A2)] There exist constants $C>0$ such that for all $x\in\br^n$, $y\in\br^m$ and $t\in\br$,
\begin{align}\label{cond:linear growth}
 |b(t,x,y)|+|f(t,x,y)|\leq C(1+|x|+|y|). 
\end{align}
\item[(A3)] There exists a constant $\Lambda>0$ such that for all $x\in\br^n$, $y\in\br^m$ and $t\in\br$,
    \begin{equation}\label{eq:bounded}
        \Vert\sigma(t,x)\Vert+\Vert g(t,y)\Vert\leq \Lambda.
    \end{equation}
\item[(A4)] There exists a constant $\lambda>q$, with $q=2^{\alpha-1} C_{\alpha,d_2}S_{d_2}((2-\alpha)^{-1}+(\alpha-1)^{-1}+(\alpha-p)^{-1})C_g^\alpha$ for some $1<p<\alpha$, such that for all $x\in\br^n$, $y_i\in\br^m$, $i=1,2$ and $t\in\br$,
\begin{equation}\label{dissipative condition}
\begin{aligned}
\langle y_1-y_2,f(t,x,y_1)-f(t,x,y_2)\rangle
&\leq-\lambda|y_1-y_2|^2.
\end{aligned}
\end{equation}
\end{description}
\end{assumption}

\begin{assumption}\label{con:periodic}
    The coefficients $b(t,x,y)$, $\sigma(t,x)$ are $\tau_1$-periodic functions in $t$, while $f(t,x,y)$, $g(t,y)$ are $\tau_2$-periodic functions in $t$, i.e. for any $x\in\br^n$, $y\in\br^m$ and $t\in\br$, the following relations hold:
    \begin{equation}\label{b sigma periodic}
        \begin{aligned}
           b(t+\tau_1,x,y)=b(t,x,y),~\sigma(t+\tau_1,x)=\sigma(t,x),\\
           f(t+\tau_2,x,y)=f(t,x,y),~g(t+\tau_2,y)=g(t,y).
        \end{aligned}
    \end{equation}
\end{assumption}

\begin{remark}\label{remark:duf is bounded}
    Condition \eqref{dissipative condition} is referred to as the dissipative condition, which ensures the contraction property and ergodicity of the frozen equation \eqref{frozen SDE} introduced below, where the parameter $q$ is a technical constant introduced in the following proofs.
\end{remark}

In the following throughout the paper, we use the notation $\kappa=(\alpha,d_1,d_2,n,m,\Lambda,C,C_\sigma,C_g,\lambda)$ as a set of parameters and 
$C_\cdot$ denotes a positive constant whose dependence is indicated by the subscript. The values of these constants may change from line to line.

We now state the main results of this paper. 
\begin{theorem}\label{thm:X_t^ve-bar{X}_t^ve}
    Suppose that Assumptions \ref{cond} and \ref{con:periodic} hold. Then there exists $C_{\kappa,p,s,T}>0$ such that for any $\ve\in(0,1)$ and $p\in(1,\alpha)$,
    \begin{equation}\label{eq:|X_t^ve-bar{X}_t^ve|^p}
        \sup_{t\in[s,T]}\be[|X_t^\ve-\bar{X}_t^\ve|^p]\leq C_{\kappa,p,s,T}(1+|x|^p+|y|^p)\ve^\frac{p-1}{\alpha+p-1}.
    \end{equation}
    Here, $\bar{X}_t^\ve$ is the unique solution to the following averaged equation
    \begin{equation}\label{thmeq:avraged equation1}
       \begin{cases}
          \d \bar{X}_t^\ve=\bar{b}(t,t/\ve,\bar{X}_t^\ve)\d t+\sigma(t,\bar{X}_t^\ve) \d L_t^1,\\
          \bar{X}_s^\ve=x,
       \end{cases}
    \end{equation}
where
\begin{equation}\label{eq:the definition of bar b}
    \bar{b}(t_1,t_2,x)=\int_{\br^m}b(t_1,x,y)\rho_{t_2}^{x}(\d y),
\end{equation}
and $\rho_t^{x}$ is the unique periodic measure associated with the frozen equation
\begin{equation*}
\begin{cases}
\d Y_t=f(t,x,Y_t)\d t+g(t,Y_t) \d L_t^2,\\
Y_s=y.
\end{cases}
\end{equation*}
\end{theorem}

In order to highlight an important observation, we first recall the definition of quasi-periodic functions, further details can be found in \cite{FQZ21}.
 \begin{definition}
     Let the reciprocals of $\tau_1>0$ and $\tau_2>0$ be rationally linearly independent. A continuous function $F(r)$ is said to be quasi-periodic with periods $\tau_1$ and $\tau_2$ if there exists a continuous function $f(r_1,r_2)$ such that $F(r)=f(r,r)$, where $f$ is periodic in $r_1$ and $r_2$ with periods $\tau_1$ and $\tau_2$, respectively.
 \end{definition}

 \begin{remark}
 We rewrite \eqref{thmeq:avraged equation1} in the following form
\begin{equation*}
\begin{cases}
\d \bar{X}_t^\ve=\bar{B}^\ve(t,\bar{X}_t^\ve)\d t+\sigma(t,\bar{X}_t^\ve)\d L_t^1,\\
\bar{X}_s^\ve=x,
\end{cases}
\end{equation*}
where $\bar{B}^\ve(t,x)=\bar{b}(t,t/\ve,x)$. 
It follows from \eqref{cond:Lip} and Assumptions \ref{con:periodic} that $\bar{b}(t_1,t_2,x)$ is continuous and periodic in $t_1$ and $t_2$ with periods $\tau_1$ and $\tau_2$, respectively. 
For any fixed $\ve$, if the reciprocals of $\tau_1$ and $\ve\tau_2$ be rationally linearly independent, $\bar{B}^\ve(t,x)$ is a quasi-periodic function of $t$ with a large period $\tau_1$ and a small period $\ve\tau_2$. 
Consequently, the averaged system has random quasi-periodicity. The quasi-periodic phenomena can be observed in reality as the parameter $\ve$ will not actually be zero, although it could be very small. 
\end{remark}

\begin{theorem}\label{thm:X_t^ve-bar X}
    Suppose that Assumptions \ref{cond} and \ref{con:periodic} hold. Then there exists $C_{\kappa,p,s,T}>0$ such that for any $\ve\in(0,1)$ and $p\in(1,\alpha)$,
    \begin{equation}
        \sup_{t\in[s,T]}\be[|X_t^\ve-\bar{X}_t|^p]\leq C_{\kappa,p,s,T}(1+|x|^p+|y|^p)\ve^\frac{p-1}{\alpha+p-1}.
    \end{equation}
     Here, $\bar{X}_t$ is the unique solution to the following averaged equation without $\ve$
    \begin{equation}\label{eq:SDEAV2}
        \begin{cases}
          \d \bar{X}_t=\bar{b}(t,\bar{X}_t)\d t+\sigma(t,\bar{X}_t)\d L_t^1,\\
          \bar{X}_s=x,
        \end{cases}
     \end{equation}
   where
   \begin{equation}\label{eq:bar b(t,x)}
       \bar{b}(t_1,x)=\frac{1}{\tau_2}\int_0^{\tau_2}\bar{b}(t_1,t_2,x)\d t_2,
   \end{equation}
   and $\bar{b}(t_1,t_2,x)$ is defined in \eqref{eq:the definition of bar b}.
\end{theorem}

\begin{remark}
   (i) To establish the fundamental mathematical framework and essential tools, we assume that $b(t,x,y)$ is $\frac{1}{\alpha}$-H\"older continuous in $t$, a natural condition for SDEs driven by $\alpha$-stable processes. This is rooted in the self-similarity of the $\alpha$-stable process, namely that $L_{ct}$ and $c^\frac{1}{\alpha}L_t$ have the same distribution, which balances the pathwise regularity of the deterministic drift and the heavy-tailed stochastic driving force.\vspace{3pt}\\
   (ii) We emphasize that our approach readily extends to cases where $b(t,x,y)$ is $\gamma$-H\"older continuous in $t$ for any $\gamma\in (0,1]$. The convergence rate in Theorems \ref{thm:X_t^ve-bar{X}_t^ve} and \ref{thm:X_t^ve-bar X} is obtained by optimizing the discretization parameter $\Delta$ in Khasminskii's time discretization method. The estimates in \eqref{estmate I1}, \eqref{estmate I2}, and \eqref{estmate I3n}, as well as those in \eqref{eq:2I1}, \eqref{eq:2I2}, and \eqref{eq:2I3}, yield the upper bound
   \[C_{\kappa,p,s,T}\Delta^{\beta}+C_{\kappa,p,s,T}\frac{\varepsilon}{\Delta},~~\beta=\min\left\{\frac{p-1}{\alpha},\gamma\right\},\]
where the first term arises from freezing the slow variable over each subinterval of length
$\Delta$, while the second represents the averaging error due to the finite relaxation time of the fast dynamics. Optimizing this upper bound by choosing $\Delta=\varepsilon^{\frac{1}{1+\beta}}$ yields the convergence rate $\varepsilon^{\frac{\beta}{1+\beta}}$.
In particular, when $b(t,x,y)$ is $\frac{1}{\alpha}$-H\"older continuous in time, we have
$\gamma=\frac{1}{\alpha}$, so that $\beta=\frac{p-1}{\alpha}$, 
which gives the convergence rate $
\varepsilon^{\frac{p-1}{\alpha+p-1}}$.
\end{remark}

We conclude this section with an important inequality that will be frequently used in the following proofs.

\begin{lemma}\label{lem:H(y,z)}
    For any $y,z\in\br^d$ and $p\in(1,\alpha)$, let $D$ be a $d\times d$ matrix. Define
    \begin{align*}
        H(y,z,D)=(\theta+|y+Dz|^2)^\frac{p}{2}-(\theta+|y|^2)^\frac{p}{2}-\langle p(\theta+|y|^2)^{\frac{p}{2}-1}y, Dz\rangle,
    \end{align*}
    where $\theta$ is a fixed positive constant. Then the following estimate holds:
    \begin{equation}\label{eq:H(y,z)}
        H(y,z,D)\leq \int_0^1\int_0^1\frac{pu_1\langle Dz,Dz\rangle}{(\theta+|y+u_2u_1Dz|^2)^{1-\frac{p}{2}}}\d u_2\d u_1.
    \end{equation}
\end{lemma}
\begin{proof}
    By the fundamental theorem of calculus, we have the second-order integral representation
    \begin{align*}
        H(y,z,D)=\int_0^1\int_0^1&\frac{pu_1\langle Dz,Dz\rangle}{(\theta+|v|^2)^{1-\frac{p}{2}}}-\frac{p(2-p)u_1\langle (v\otimes v)Dz,Dz\rangle}{(\theta+|v|^2)^{2-\frac{p}{2}}}\d u_2\d u_1,
    \end{align*}
    where $v=y+u_2u_1Dz$ and  $\otimes$ is the tensor product. Since $v\otimes v$ is a semi-positive definite symmetric matrix, it follows that $\langle (v\otimes v)Dz,Dz\rangle\geq 0$,
    which implies \eqref{eq:H(y,z)}.
\end{proof}

\section{Proof of the main results}\label{sec:Proof of the main results}
The proofs of Theorems \ref{thm:X_t^ve-bar{X}_t^ve} and \ref{thm:X_t^ve-bar X} are provided in this section, and the arguments are divided into the following steps. The first step is to establish preliminary estimates for the solution $(X_t^\varepsilon, Y_t^\varepsilon)$ of the multiscale coupled system. Based on Khasminskii's discretization method, we then construct an auxiliary process $\hat{Y}_t^\varepsilon$ and establish some key estimates. Next, we prove the ergodicity of the time-inhomogeneous frozen equation. This property is crucial for the construction of the averaged equation. Finally, combining the preceding estimates and the ergodicity result, we complete the proof of the main theorems.

\subsection{A priori estimates for the processes $(X_t^\varepsilon,Y_t^\varepsilon)$}\label{subsec:Some prior estimates}
This subsection begins by proving the existence and uniqueness of the solution for the multiscale systems \eqref{SDEs}. We then establish uniform moment estimates for its solution with respect to $\ve$, along with an analysis of the dependence on time for the slow variable $X_t^\ve$.
\begin{theorem}
    Suppose that Assumption \ref{cond} holds. Then for any fixed $\ve\in(0,1)$, the slow–fast SDE \eqref{SDEs} admits a unique solution $\{(X_t^\varepsilon,Y_t^\varepsilon), t\ge s\}$ for every initial value $(x,y)\in\mathbb{R}^n\times\mathbb{R}^m$.
\end{theorem}

\begin{proof}
The proof of this theorem follows a standard argument, so we omit the details here. For more details, see  \cite[Theorem 6.2.9]{A09}.
\end{proof}

\begin{lemma}\label{lem:sup E[|X|] and sup E[|Y|]}
    Suppose that Assumption \ref{cond} holds. Then there exists $C_{\kappa,p,s,T}>0$ such that for all $p\in[1,\alpha)$,
    \begin{equation}
        \sup_{\ve\in(0,1)}\sup_{t\in[s,T]}\be[|X_t^\ve|^p]\leq C_{\kappa,p,s,T}(1+|x|^p+|y|^p),
    \end{equation}
    and
    \begin{equation}
        \sup_{\ve\in(0,1)}\sup_{t\in[s,T]}\be[|Y_t^\ve|^p]\leq C_{\kappa,p,s,T}(1+|x|^p+|y|^p).
    \end{equation}
\end{lemma}

\begin{proof}
    We consider only the case of $p\in(1,\alpha)$, as the case of $p=1$ is similar. Applying It\^o's formula, we derive that
    \begin{align*}
        \be[(1+|X_t^\ve|^2)^{\frac{p}{2}}]&=(1+|x|^2)^{\frac{p}{2}}+p\int_s^t\be[\langle(1+|X_r^\ve|^2)^{\frac{p}{2}-1}X_r^\ve,b(r,X_r^\ve,Y_r^\ve)\rangle]\d r\\
        &\quad+\int_s^t\int_{|z|>1}\be[(1+|X_r^\ve+\sigma(r,X_r^\ve)z|^2)^{\frac{p}{2}}-(1+|X_r^\ve|^2)^{\frac{p}{2}}]\nu_1(\d z)\d r\\
        &\quad+\int_s^t\int_{|z|\leq 1}\be[(1+|X_r^\ve+\sigma(r,X_r^\ve)z|^2)^{\frac{p}{2}}-(1+|X_r^\ve|^2)^{\frac{p}{2}}\\
        &\qquad\qquad\qquad\qquad\qquad-p\langle(1+|X_r^\ve|^2)^{\frac{p}{2}-1}X_r^\ve,\sigma(r,X_r^\ve)z\rangle]\nu_1(\d z)\d r\\
        &=:(1+|x|^2)^{\frac{p}{2}}+I_1+I_2+I_3.
    \end{align*}
    For the term $I_1$, we apply Young's inequality and \eqref{cond:linear growth} to obtain
    \begin{align*}
        I_1&\leq p\int_s^t\be[(1+|X_r^\ve|^2)^\frac{p-1}{2}|b(r,X_r^\ve,Y_r^\ve)|]\d r\\
        &\leq p\int_s^t\be[(1+|X_r^\ve|^2)^\frac{p}{2}]\d r+p^{1-p}(p-1)^{p-1}\int_s^t\be[|b(r,X_r^\ve,Y_r^\ve)|^p]\d r\\
        &\leq p\int_s^t\be[(1+|X_r^\ve|^2)^\frac{p}{2}]\d r+C_{\kappa,p}\int_s^t\be[1+|X_r^\ve|^p+|Y_r^\ve|^p]\d r.
    \end{align*}
    Similarly, for the term $I_2$, we obtain
    \begin{align*}
        I_2&= p\int_s^t\int_{|z|>1}\int_0^1\be\left[\frac{\langle X_r^\ve+u\sigma(r,X_r^\ve)z,\sigma(r,X_r^\ve)z\rangle}{(1+|X_r^\ve+u\sigma(r,X_r^\ve)z|^2)^{1-\frac{p}{2}}}\right]\d u\nu_1(\d z)\d r\\
        &\leq p\int_s^t\int_{|z|>1}\int_0^1\be[|\sigma(r,X_r^\ve)z||X_r^\ve+u\sigma(r,X_r^\ve)z|^{p-1}]\d u\nu_1(\d z)\d r\\
        &\leq C_p\int_s^t\int_{|z|>1}\int_0^1\be[|X_r^\ve|^{p-1}\Vert\sigma(r,X_r^\ve)\Vert|z|+u^{p-1}|\sigma(r,X_r^\ve)z|^p]\d u\nu_1(\d z)\d r\\
        &\leq C_p\int_s^t\int_{|z|>1}\be[|X_r^\ve|^p|z|+\Vert\sigma(r,X_r^\ve)\Vert^p |z|+\Vert\sigma(r,X_r^\ve)\Vert^p|z|^p]\nu_1(\d z)\d r\\
        &\leq C_pC_{\alpha,d_1}S_{d_1}(\alpha-1)^{-1}\int_s^t\be[(1+|X_r^\ve|^2)^\frac{p}{2}]\d r+C_{\kappa,p}(t-s).
    \end{align*}
    To estimate $I_3$, applying Lemma \ref{lem:H(y,z)} together with \eqref{eq:bounded} yields
    \begin{align*}
        I_3&\leq p\int_s^t\int_{|z|\leq 1}\int_0^1\int_0^1\be\left[\frac{u_1\langle\sigma(r,X_r^\ve)z,\sigma(r,X_r^\ve)z\rangle}{(1+|X_r^\ve+u_2u_1\sigma(r,X_r^\ve)z|^2)^{1-\frac{p}{2}}}\right]\d u_2\d u_1\nu_1(\d z)\d r\\
        &\leq 2^{-1}p\int_s^t\int_{|z|\leq 1}\be[|\sigma(r,X_r^\ve)z|^2]\nu_1(\d z)\d r\\
        &\leq C_{\kappa,p}(t-s).
    \end{align*}
    A combination of the above estimates yields
    \begin{equation}\label{eq:sup E[|X_^ve|^p]}
        \sup_{t\in[s,T]}\be[|X_t^\ve|^p]\leq C_{\kappa,p}(1+|x|^p) +C_{\kappa,p}\int_s^T\be[|X_r^\ve|^p]\d r+C_{\kappa,p}\int_s^T\be[|Y_r^\ve|^p]\d r+C_{\kappa,p,s,T}.
    \end{equation}
    We apply It\^o's formula again to get
    \begin{align*}
        \be[(1+|Y_t^\ve|^2)^{\frac{p}{2}}]&=(1+|y|^2)^{\frac{p}{2}}+\ve^{-1}p\int_s^t\be[\langle(1+|Y_r^\ve|^2)^{\frac{p}{2}-1}Y_r^\ve,f(r/\ve,X_r^\ve,Y_r^\ve)\rangle]\d r\\
        &~~~+\int_s^t\int_{|z|>\ve^{\frac{1}{\alpha}}C_g^\frac{\alpha}{1-\alpha}}\be[(1+|Y_r^\ve+\ve^{-\frac{1}{\alpha}}g(r/\ve,Y_r^\ve)z|^2)^{\frac{p}{2}}-(1+|Y_r^\ve|^2)^{\frac{p}{2}}]\nu_2(\d z)\d r\\
        &~~~+\int_s^t\int_{|z|\leq \ve^{\frac{1}{\alpha}}C_g^\frac{\alpha}{1-\alpha}}\be[(1+|Y_r^\ve+\ve^{-\frac{1}{\alpha}}g(r/\ve,Y_r^\ve)z|^2)^{\frac{p}{2}}-(1+|Y_r^\ve|^2)^{\frac{p}{2}}\\
        &\qquad\qquad\qquad\qquad\qquad\qquad\quad-p\langle(1+|Y_r^\ve|^2)^{\frac{p}{2}-1}Y_r^\ve,\ve^{-\frac{1}{\alpha}}g(r/\ve,Y_r^\ve)z\rangle]\nu_2(\d z)\d r\\
        &=:(1+|y|^2)^{\frac{p}{2}}+I_4+I_5+I_6.
    \end{align*}
    For the term $I_4$, it follows from \eqref{dissipative condition} and \eqref{cond:linear growth} that
    \begin{align*}
        I_4&\leq \ve^{-1}p\int_s^t\be\left[\frac{\langle Y_r^\ve-0,f(r/\ve,X_r^\ve,Y_r^\ve)-f(r/\ve,X_r^\ve,0)\rangle}{(1+|Y_r^\ve|^2)^{1-\frac{p}{2}}}+\frac{\langle Y_r^\ve,f(r/\ve,X_r^\ve,0)\rangle}{(1+|Y_r^\ve|^2)^{1-\frac{p}{2}}}\right]\d r\\
        &\leq \ve^{-1}p\int_s^t\be\left[\frac{-\lambda|Y_r^\ve|^2}{(1+|Y_r^\ve|^2)^{1-\frac{p}{2}}}+\frac{|f(r/\ve,X_r^\ve,0)|}{(1+|Y_r^\ve|^2)^{\frac{1}{2}-\frac{p}{2}}}\right]\d r\\
        &\leq -\ve^{-1}p(\lambda-\varsigma)\int_s^t\be[(1+|Y_r^\ve|^2)^\frac{p}{2}]\d r+\ve^{-1}p\lambda\int_s^t\be[(1+|Y_r^\ve|^2)^{\frac{p}{2}-1}]\d r\\
        &\quad+\ve^{-1}\varsigma^{1-p}(p-1)^{p-1}p^{1-p}\int_s^t\be[|f(r/\ve,X_r^\ve,0)|^p]\d r\\
        &\leq-\ve^{-1}p(\lambda-\varsigma)\int_s^t\be[(1+|Y_r^\ve|^2)^\frac{p}{2}]\d r+\ve^{-1}\varsigma^{1-p}C_{\kappa,p}\int_s^t\be[1+|X_r^\ve|^p]\d r+\ve^{-1}C_{\kappa,p}(t-s),
    \end{align*}
    where $\varsigma>0$, and its explicit value will be provided below. Following the same approach as in the estimates of $I_2$ and $I_3$, we find that
    \begin{align*}
        I_5&= p\int_s^t\int_{|z|>\ve^\frac{1}{\alpha}C_g^\frac{\alpha}{1-\alpha}}\int_0^1\be\left[\frac{\langle Y_r^\ve+\ve^{-\frac{1}{\alpha}}ug(r/\ve,Y_r^\ve)z,\ve^{-\frac{1}{\alpha}}g(r/\ve,Y_r^\ve)z\rangle}{(1+|Y_r^\ve+\ve^{-\frac{1}{\alpha}}ug(r/\ve,Y_r^\ve)z|^2)^{1-\frac{p}{2}}}\right]\d u\nu_2(\d z)\d r\\
        &\leq \ve^{-\frac{1}{\alpha}}p\int_s^t\int_{|z|>\ve^\frac{1}{\alpha}C_g^\frac{\alpha}{1-\alpha}}\int_0^1\be[|g(r/\ve,Y_r^\ve)z||Y_r^\ve+\ve^{-\frac{1}{\alpha}}ug(r/\ve,Y_r^\ve)z|^{p-1}]\d u\nu_2(\d z)\d r\\
        &\leq \ve^{-\frac{1}{\alpha}}p\int_s^t\int_{|z|>\ve^\frac{1}{\alpha}C_g^\frac{\alpha}{1-\alpha}}\be[\Vert g(r/\ve,Y_r^\ve)\Vert|z||Y_r^\ve|^{p-1}+\ve^{-\frac{p-1}{\alpha}}|g(r/\ve,Y_r^\ve)z|^p]\nu_2(\d z)\d r\\
        &\leq  \ve^{-\frac{1}{\alpha}}p\int_s^t\int_{|z|>\ve^\frac{1}{\alpha}C_g^\frac{\alpha}{1-\alpha}}\be[|Y_r^\ve|^p|z|+p^{-p}(p-1)^{p-1}\Vert g(r/\ve,Y_r^\ve)\Vert^p|z|]\nu_2(\d z)\d r\\
        &\quad+\ve^{-\frac{1}{\alpha}}p\int_s^t\int_{|z|>\ve^\frac{1}{\alpha}C_g^\frac{\alpha}{1-\alpha}}\be[\ve^{-\frac{p-1}{\alpha}}\Vert g(r/\ve,Y_r^\ve)\Vert^p|z|^p]\nu_2(\d z)\d r\\
        &\leq \ve^{-1}pC_{\alpha,d_2}S_{d_2}(\alpha-1)^{-1}C_g^\alpha\int_s^t\be[(1+|Y_r^\ve|^2)^\frac{p}{2})]\d r+\ve^{-1}C_{\kappa,p}(t-s),
    \end{align*}
    and
    \begin{align*}
        I_6&\leq p\int_s^t\int_{|z|\leq \ve^\frac{1}{\alpha}C_g^\frac{\alpha}{1-\alpha}}\int_0^1\int_0^1\be\left[\frac{u_1\langle \ve^{-\frac{1}{\alpha}}g(r/\ve,Y_r^\ve)z,\ve^{-\frac{1}{\alpha}}g(r/\ve,Y_r^\ve)z\rangle}{(1+|Y_r^\ve+u_2u_1\ve^{-\frac{1}{\alpha}}g(r/\ve,Y_r^\ve)z|^2)^{1-\frac{p}{2}}}\right]\d u_2\d u_1\nu_2(\d z)\d r\\
        &\leq \ve^{-\frac{2}{\alpha}}2^{-1}p\int_s^t\int_{|z|\leq \ve^\frac{1}{\alpha}C_g^\frac{\alpha}{1-\alpha}}\be[|g(r/\ve,Y_r^\ve)z|^2]\nu_2(\d z)\d r\\
        &\leq \ve^{-1}C_{\kappa,p}C_{\alpha,d_2}S_{d_2}(2-\alpha)^{-1}C_g^\frac{(2-\alpha)\alpha}{1-\alpha}(t-s).
    \end{align*}
    With the choice of $\varsigma=(2^{\alpha-1}-1)C_{\alpha,d_2}S_{d_2}(\alpha-1)^{-1}C_g^\alpha$, the above estimates together give
    \begin{align*}
        \be[|Y_t^\ve|^p]&\leq (1+|y|^2)^\frac{p}{2}-\ve^{-1}p(\lambda-q')\int_s^t\be[(1+|Y_r^\ve|^2)^\frac{p}{2}]\d r\\
        &\quad+\ve^{-1}C_{\kappa,p}\int_s^t\be[1+|X_r^\ve|^p]\d r+\ve^{-1}C_{\kappa,p}(t-s),
    \end{align*}
    where 
    \[q'= 2^{\alpha-1}C_{\alpha,d_2}S_{d_2}(\alpha-1)^{-1}C_g^\alpha<q,\]
    with $q$ defined in Assumption \ref{cond}. Therefore, we derive
    \begin{align*}
        \frac{\d}{\d t}\be[|Y_t^\ve|^p]&\leq -\ve^{-1}p(\lambda-q)\be[|Y_t^\ve|^p]+\ve^{-1}C_{\kappa,p}(1+\be[|X_t^\ve|^p]).
    \end{align*}
    Applying the comparison theorem yields
    \begin{equation}\label{eq:E[|Y_t^ve|^p]}
        \begin{aligned}
        \be[|Y_t^\ve|^p]&\leq \re^{-\frac{p(\lambda-q)(t-s)}{\ve}}|y|^p+\frac{C_{\kappa,p}}{\ve}\int_s^t\re^{-\frac{p(\lambda-q')(t-r)}{\ve}}(1+\be[|X_r^\ve|^p])\d r\\
        &\leq |y|^p+C_{\kappa,p,s,T}\left(1+\sup_{s\leq r\leq t}\be[|X_r^\ve|^p]\right).
    \end{aligned}
    \end{equation}
    Combining the above estimates \eqref{eq:sup E[|X_^ve|^p]} and \eqref{eq:E[|Y_t^ve|^p]}, we have
    \begin{align*}
        \sup_{s\in[t,T]}\be[|X_t^\ve|^p]&\leq C_{\kappa,p,s,T}(1+|x|^p+|y|^p)+C_{\kappa,p,s,T}\int_s^T\sup_{s\leq u\leq r}\be[|X_u|^p]\d r.
    \end{align*}
    Using Grownall's inequality yields
    \[\sup_{s\in[t,T]}\be[|X_t^\ve|^p]\leq C_{\kappa,p,s,T}(1+|x|^p+|y|^p),\]
    which also gives
    \[\sup_{s\in[t,T]}\be[|Y_t^\ve|^p]\leq C_{\kappa,p,s,T}(1+|x|^p+|y|^p).\]
    Hence, the proof is concluded.
\end{proof}

\begin{lemma}\label{lem:X(t+h)-X(t)}
    Suppose that Assumption \ref{cond} holds. Then for any $s<t<t+\Delta\leq T$ with a given $T>s$ and $\Delta> 0$, there exists $C_{\kappa,p,s,T}>0$ such that for all $p\in[1,\alpha)$,
    \begin{equation}\label{eq:X(t+h)-X(t)}
        \be[|X_{t+\Delta}^\ve-X_t^\ve|^p]\leq C_{\kappa,p,s,T}(1+|x|^p+|y|^p)(\Delta^p+\Delta^{\frac{p}{\alpha}})
    \end{equation}
\end{lemma}

\begin{proof}
    Notice that
    \[X_{t+\Delta}^\ve-X_t^\ve=\int_t^{t+\Delta}b(r,X_r^\ve,Y_r^\ve)\d r+\int_t^{t+\Delta}\sigma(r,X_r^\ve)\d L_r^1.\]
    Let $\mathbb{L}$ denote the second term on the right-hand side of the above equality. By L\'evy-It\^o's decomposition, BDG's inequality and Kunita's inequality, it follows that
    \begin{align*}
        \be[|\mathbb{L}|^p]&\leq C_p\be\left[\left|\int_t^{t+\Delta}\int_{|z|\leq\Delta^\frac{1}{\alpha}}\sigma(r,X_r^\ve)z\tilde{N}(\d z,\d r)\right|^p\right]+C_p\be\left[\left|\int_t^{t+\Delta}\int_{|z|>\Delta^\frac{1}{\alpha}}\sigma(r,X_r^\ve)zN(\d z,\d r)\right|^p\right]\\
        &\leq C_p\be\left[\int_t^{t+\Delta}\int_{|z|\leq\Delta^\frac{1}{\alpha}}|\sigma(r,X_r^\ve)z|^2\nu_1(\d z)\d r\right]^\frac{p}{2}+C_p\be\left[\int_t^{t+\Delta}\int_{|z|>\Delta^\frac{1}{\alpha}}|\sigma(r,X_r^\ve)z|^p\nu_1(\d z)\d r\right]\\
        &\leq C_{\kappa,p}\Delta^\frac{p}{\alpha}.
    \end{align*}
    Then we derive from the H\"older's inequality, \eqref{eq:bounded} and Lemma \ref{lem:sup E[|X|] and sup E[|Y|]} that
    \begin{align*}
        \be[|X_{t+\Delta}^\ve-X_t^\ve|^p]&\leq \Delta^{p-1}\int_t^{t+\Delta}\be[|b(r,X_r^\ve,Y_r^\ve)|^p]\d r+\be\left[\left|\int_t^{t+\Delta}\sigma(r,X_r^\ve)\d L_r^1\right|^p\right]\\
        &\leq C\Delta^{p-1}\int_t^{t+\Delta}\be[1+|X_r^\ve|^p+|Y_r^\ve|^p]\d r+C_{\kappa,p}\Delta^\frac{p}{\alpha}\\
        &\leq C_{\kappa,p,s,T}(1+|x|^p+|y|^p)\Delta^p+C_{\kappa,p}\Delta^{\frac{p}{\alpha}}.
    \end{align*}
    This finishes the proof.
\end{proof}

\subsection{The auxiliary process $\hat{Y}_t^\ve$ for the fast subsystem $Y_t^\ve$}\label{subsec:The auxiliary process hat{Y}_t^ve for the fast system}
Based on Khasminskii's time-discretization approach \cite{K68}, we construct an auxiliary process $\hat{Y}_t^\ve\in\br^m$. Choosing a step size $\Delta=\frac{\tau_2}{N}$ for some $N\in\bn$, where $\Delta$ depends on $\ve$, and we will provide its specific value below. For any $t\in[s+k\Delta, (s+(k+1)\Delta)\wedge T]$, we define the process $\hat{Y}_t^\ve$ with initial value $\hat{Y}_s^\ve=Y_s^\ve=y$ by 
\begin{equation*}
    \hat{Y}_t^\ve=\hat{Y}_{s+k\Delta}^\ve+\ve^{-1}\int_{s+k\Delta}^tf(r/\ve,X_{s+k\Delta}^\ve,\hat{Y}_r^\ve)\d r+\ve^{-\frac{1}{\alpha}}\int_{s+k\Delta}^tg(r/\ve,\hat{Y}_r^\ve)\d L_r^2,
\end{equation*}
which is equivalently written as
\begin{equation*}
    \hat{Y}_t^\ve=y+\ve^{-1}\int_s^tf(r/\ve,X_{r_\Delta}^\ve,\hat{Y}_r^\ve)\d r+\ve^{-\frac{1}{\alpha}}\int_s^tg(r/\ve,\hat{Y}_r^\ve)\d L_r^2,
\end{equation*}
where $t_\Delta=s+\lfloor (t-s)/\Delta\rfloor\Delta$, and $\lfloor\cdot\rfloor$ denotes the floor function.

\begin{lemma}
    Suppose that Assumption \ref{cond} holds. Then there exists $C_{\kappa,p,s,T}>0$ such that for all $p\in[1,\alpha)$,
    \begin{equation}
        \sup_{\ve\in(0,1)}\sup_{t\in[s,T]}\be[|\hat{Y}_t^\ve|^p]\leq C_{\kappa,p,s,T}(1+|x|^p+|y|^p).
    \end{equation}
\end{lemma}

\begin{proof}
    The details are omitted here by analogy with the proof of Lemma \ref{lem:sup E[|X|] and sup E[|Y|]}.
\end{proof}

\begin{lemma}\label{lem:Y(t)-hatY(t)}
    Suppose that Assumption \ref{cond} holds. Then there exists $C_{\kappa,p,s,T}>0$ such that for all $p\in[1,\alpha)$,
    \begin{equation}
        \sup_{\ve\in(0,1)}\sup_{t\in[s,T]}\be[|Y_t^\ve-\hat{Y}_t^\ve|^p]\leq C_{\kappa,p,s,T}(1+|x|^p+|y|^p)(\Delta^p+\Delta^\frac{p}{\alpha}).
    \end{equation}
\end{lemma}

\begin{proof}
    We consider only the case of $p\in(1,\alpha)$, as the case of $p=1$ is similar. For simplicity, we set $Z_t^\ve:= Y_t^\ve-\hat{Y}_t^\ve$, $\bar{f}_t=f(t/\ve,X_t^\ve,Y_t^\ve)-f(t/\ve,X_{t_\Delta}^\ve,\hat{Y}_t^\ve)$ and $\hat{g}_t:=g(t/\ve,Y_t^\ve)-g(t/\ve,\hat{Y}_t^\ve)$. Notice that
    \begin{align*}
        Z_t^\ve&=\ve^{-1}\int_s^t\bar{f}_r\d r+\ve^{-\frac{1}{\alpha}}\int_s^t\hat{g}_r\d L_r^2.
    \end{align*}
     For any fixed $\theta>0$ and any $p\in(1,\alpha)$, let $V(x)=(\theta+|x|^2)^\frac{p}{2}$. Applying It\^o's formula then gives
    \begin{align*}\label{eq:Lpp 1}
    \be[V(Z_t^\ve)]&=\theta^{\frac{p}{2}}+\ve^{-1}p\int_s^t\be[(\theta+|Z_r^\ve|^2)^{\frac{p}{2}-1}\langle Z_r^\ve,\bar{f}_r\rangle]\d r\\
    &\quad+\int_s^t\int_{|z|>\frac{\ve^\frac{1}{\alpha}}{2C_g}}\be[V(Z_r^\ve+\ve^{-\frac{1}{\alpha}}\hat{g}_rz)-V(Z_r^\ve)]\nu_2(\d z)\d r\\
    &\quad+\int_s^t\int_{|z|\leq \frac{\ve^\frac{1}{\alpha}}{2C_g}}\be[V(Z_r^\ve+\ve^{-\frac{1}{\alpha}}\hat{g}_rz)-V(Z_r^\ve)-\langle \nabla V(Z_r^\ve),\ve^{-\frac{1}{\alpha}}\hat{g}z\rangle]\nu_2(\d z)\d r\\
    &=:\theta^{\frac{p}{2}}+I_1+I_2+I_3.
    \end{align*}
    For the first term, we apply \eqref{dissipative condition}, Young's inequality, and \eqref{cond:Lip} to obtain that
    \begin{align*}
        I_1&=\ve^{-1}p\int_s^t\be[(\theta+|Z_r^\ve|^2)^{\frac{p}{2}-1}(\langle Z_r^\ve,\hat{f}_r\rangle+\langle Z_r^\ve,\tilde{f}_r\rangle)]\d r\\
        &\leq \ve^{-1}p\int_s^t\be[-\lambda|Z_r^\ve|^2(\theta+|Z_u^\ve|^2)^{\frac{p}{2}-1}+|\tilde{f}_r|(\theta+|Z_r^\ve|^2)^\frac{p-1}{2}]\d r\\
        &\leq -\ve^{-1}p\lambda\int_s^t\be[V(Z_r^\ve)]\d r+\ve^{-1}p\lambda\theta^\frac{p}{2}(t-s)\\
        &\quad+\ve^{-1}C_p\varsigma^{1-p}\int_s^t\be[|\tilde{f}_r|^p]\d r+\ve^{-1}p\varsigma \int_s^t\be[V(Z_r^\ve)]\d r\\
        &\leq -\ve^{-1}p(\lambda-\varsigma)\int_s^t\be[V(Z_r^\ve)]\d r+\ve^{-1}p\lambda\theta^\frac{p}{2}(t-s)\\
        &\quad+\ve^{-1}C_{\kappa,p}\varsigma^{1-p}\int_s^t\be[|X_r^\ve-X_{r_\Delta}^\ve|^p]\d r,
    \end{align*}
    where $\hat{f}_t:=f(t/\ve,X_t^\ve,Y_t^\ve)-f(t/\ve,X_t^\ve,\hat{Y}_t^\ve)$ and $\tilde{f}_t=f(t/\ve,X_t^\ve,\hat{Y}_t^\ve)-f(t/\ve,X_{t_\Delta}^\ve,\hat{Y}_t^\ve)$. For the second term, by \eqref{cond:Lip} again, we obtain
    \begin{align*}
        I_2&= p\int_s^t\int_{|z|>\frac{\ve^\frac{1}{\alpha}}{2C_g}}\int_0^1\be\left[\frac{\langle Z_r^\ve+u\ve^{-\frac{1}{\alpha}}\hat{g}z,\ve^{-\frac{1}{\alpha}}\hat{g}z\rangle}{(\theta+|Z_r^\ve+u\ve^{-\frac{1}{\alpha}}\hat{g}z|^2)^{1-\frac{p}{2}}}\right]\d u\nu_2(\d z)\d r\\
        &\leq\ve^{-\frac{1}{\alpha}}p\int_s^t\int_{|z|>\frac{\ve^\frac{1}{\alpha}}{2C_g}}\int_0^1\be[|\hat{g}z||Z_r^\ve+u\ve^{-\frac{1}{\alpha}}\hat{g}z|^{p-1}]\d u\nu_2(\d z)\d r\\
        &\leq \ve^{-\frac{1}{\alpha}}p\int_s^t\int_{|z|>\frac{\ve^\frac{1}{\alpha}}{2C_g}}\be[\Vert \hat{g}\Vert |Z_r^\ve|^{p-1}|z|+\ve^{-\frac{p-1}{\alpha}}|\hat{g}z|^p]\nu_2(\d z)\d r\\
        &\leq \ve^{-\frac{1}{\alpha}}p\int_s^t\int_{|z|>\frac{\ve^\frac{1}{\alpha}}{2C_g}}\be[C_g|Z_r^\ve|^p|z|+\ve^{-\frac{p-1}{\alpha}}C_g^p|Z_r^\ve|^p|z|^p]\nu_2(\d z)\d r\\
        &\leq \ve^{-1}pC_{\alpha,d_2}S_{d_2}(2^{\alpha-1}(\alpha-1)^{-1}+2^{\alpha-p}(\alpha-p)^{-1})C_g^\alpha\int_s^t\be[V(Z_r^\ve)]\d r.
    \end{align*}
    For the third term, Lemma \ref{lem:H(y,z)} and \eqref{cond:Lip} yield
    \begin{align*}
    I_3&\leq\int_s^t\int_{|z|<\frac{\ve^\frac{1}{\alpha}}{2C_g}}\int_0^1\int_0^1\be\left[\frac{p\langle u_1\ve^{-\frac{1}{\alpha}}\hat{g}z,\ve^{-\frac{1}{\alpha}}\hat{g}z\rangle}{(\theta+|Z_r^\ve+u_2u_1\ve^{-\frac{1}{\alpha}}\hat{g}z|^2)^{1-\frac{p}{2}}}\right]\d u_2\d u_1\nu_2(\d z)\d r\\
    &\leq \ve^{-\frac{2}{\alpha}}p\int_s^t\int_{|z|\leq\frac{\ve^\frac{1}{\alpha}}{2C_g}}\be[\Vert \hat{g}\Vert^2|z|^2(\theta+1/4|Z_r^\ve|^2)^{\frac{p}{2}-1}\b1_{\{|Z_r^\ve|>0\}}]\nu_2(\d z)\d r\\
    &\leq\ve^{-\frac{2}{\alpha}}pC_g^2\int_s^t\int_{|z|\leq\frac{\ve^\frac{1}{\alpha}}{2C_g}}\be[ 2^{2-p}|Z_r^\ve|^p+4\theta^\frac{p}{2}]|z|^2\nu_2(\d z)\d r\\
    &\leq 2^{\alpha-p}\ve^{-1}p C_{\alpha,d_2}S_{d_2}(2-\alpha)^{-1}C_g^\alpha\int_s^t\be[V(Z_r^\ve)]\d r+2^\alpha\ve^{-1}pC_g^\alpha\theta^\frac{p}{2}(t-s).
\end{align*}
    Combining the above estimates with $\varsigma=C_{\alpha,d_2}S_{d_2}(2^{\alpha-1}-2^{\alpha-p})(\alpha-p)^{-1}C_g^\alpha$ and letting $\theta$ tend to $0$, we get
    \begin{align*}
        \frac{\d}{\d t}\be[|Z_t^\ve|^p] &\leq -\ve^{-1}p(\lambda-q')\be[|Z_t^\ve|^p]+\ve^{-1}C_{\kappa,p}\be[|X_t^\ve-X_{t_\Delta}^\ve|^p],
    \end{align*}
where 
\[q'=C_{\alpha,d_2}S_{d_2}(2^{\alpha-1}(\alpha-1)^{-1}+2^{\alpha-1}(\alpha-p)^{-1}+2^{\alpha-p}(2-\alpha)^{-1})C_g^\alpha<q.\]
Using the comparison theorem and Lemma \ref{lem:X(t+h)-X(t)}, it follows that
\begin{align*}
    \be[|Y_t^\ve-\hat{Y}_t^\ve|^p]&\leq \ve^{-1}C_{\kappa,p}\int_s^t\re^{-\frac{p(\lambda-q)(t-s)}{\ve}}\be[|X_r^\ve-X_{r_\Delta}^\ve|^p]\d r\\
    &\leq C_{\kappa,p,s,T}(1+|x|^p+|y|^p)(\Delta^p+\Delta^{\frac{p}{\alpha}}).
\end{align*}
This completes the proof.
\end{proof}

\subsection{Ergodicity for the time-inhomogeneous frozen equation}\label{subsec:Ergodicity for the fast system with frozen slow variable}
For any fixed slow variable $x\in\br^n$, we introduce the time-inhomogeneous frozen equation
\begin{equation}\label{frozen SDE}
\begin{cases}
\d Y_t=f(t,x,Y_t)\d t+g(t,Y_t) \d L_t^2,\\
Y_s=y.
\end{cases}
\end{equation}
Under Assumption \ref{cond}, it is well known that there exists a unique solution to the system \eqref{frozen SDE}, denoted by $Y_t^{s,y,x}$. Let $P^{x}(s,t)$ be the two parameter Markov semigroup corresponding to the process $Y_t^{s,y,x}$. Specifically, for any $t\geq s$, the operator is defined by
\[P^{x}(s,t)\vp(y):=\be[\vp(Y_t^{s,y,x})],\]
where $\vp:\br^m\to\br$ is a bounded measurable function. Let $\cp(\br^m)$ be the space of probability measures on $(\br^m,\cb(\br^m))$. For any $1\leq p<\alpha$, define
\[\cm_p:=\left\{\rho_t^{x}:\br\to\cp(\br^m)\left|~\sup_{t\in\br}\int_{\br^m}|w|^p\rho_t^x(\d w)<+\infty\right.\right\}.\]
We now introduce the 1-Wasserstein distance to measure the distance between two elements of $\cp(\br^m)$, i.e., for any $\mu_1,\mu_2\in\cm_1$, 
\begin{equation*}
    \cw_1(\mu_1,\mu_2)=\inf_{\gamma\in\cd(\mu_1,\mu_2)}\int_{\br^m\times\br^m}|x-y|\gamma(\d x,\d y),
\end{equation*}
where $\cd(\mu_1,\mu_2)$is the set of all joint probability measures with marginals $\mu_1$ and $\mu_2$. Its dual form (Kantorovich-Rubinstein duality) is defined by
\begin{align}
    \cw_1(\mu_1,\mu_2)=\sup_{h\in {Lip}_1}\left|\int_{\br^m}h(x)\mu_1(\d x)-\int_{\br^m}h(x)\mu_2(\d x)\right|,
\end{align}
where $Lip_1:=\{h: \br^m\to\br^m||h(x)-h(y)|\leq |x-y|\}$. To this end, we make use of the notion of periodic measures to describe the long-time behaviour of the time-inhomogeneous frozen equation. A detailed discussion of this concept can be found in \cite{FZ20}.

\begin{definition}
    A measure-valued function $\rho_t^{x}$ is said to be a $\tau$-periodic measure for the two parameter Markov semigroup $P^{x}(s,t)$ if it satisfies
    \begin{align*}
        \rho_{t+\tau}^{x}=\rho_t^{x}~~~\mbox{and}~~~P^{*,x}(s,t)\rho_s^{x}=\int_{\br^m}P^x(s,w,t,\cdot)\rho_s^x(\d w)=\rho_t^{x}.
    \end{align*}
\end{definition}

\begin{lemma}\label{lem:sup E[Y_t]}
    Suppose that Assumption \ref{cond} holds. Then for any $x\in\br^n$, $y\in\br^m$ and $s,t\in\br$, there exists $C_{\kappa,p}>0$ such that for all $p\in[1,\alpha)$,
    \begin{equation}
        \be[|Y_t^{s,y,x}|^p]\leq C_{\kappa,p}(1+|x|^p+\re^{-p(\lambda-q)(t-s)}|y|^p),
    \end{equation}
\end{lemma}

\begin{proof}
    We consider only the case of $p\in(1,\alpha)$, as the case of $p=1$ is similar. Define $V(y)=(1+|y|^2)^\frac{p}{2}$. Applying It\^o's formula yields
\begin{align*}
    \be[V(Y_t^{s,y,x})]&=(1+|y|^2)^{\frac{p}{2}}+p\int_s^t\be\left[\frac{\langle Y_r^{s,y,x},f(r,x,Y_r^{s,y,x})\rangle}{(1+|Y_r^{s,y,x}|^2)^{1-\frac{p}{2}}}\right]\d r\\
    &\quad+\int_s^t\int_{|z|>1}\be[V(Y_r^{s,y,x}+g(r,Y_r^{s,y,x})z)-V(Y_r^{s,y,x})]\nu_2(\d z)\d r\\
    &\quad+\int_s^t\int_{|z|\leq1}\be[V(Y_{r}^{s,y,x}+g(r,Y_r^{s,y,x})z)-V(Y_r^{s,y,x})\\
    &\qquad\qquad\qquad\qquad\qquad~-\langle \nabla V(Y_r^{s,y,x}),g(r,Y_r^{s,y,x})z\rangle]\nu_2(\d z)\d r\\
    &=:(1+|y|^2)^{\frac{p}{2}}+I_1+I_2+I_3.
\end{align*}
    For the term $I_1$, an application of \eqref{dissipative condition}, Young's inequality and \eqref{cond:linear growth} yields
    \begin{align*}
        I_1&\leq p\int_s^t\be\left[\frac{\langle Y_r^{s,y,x}-0,f(r,x,Y_r^{s,y,x})-f(r,x,0)\rangle}{(1+|Y_r^{s,y,x}|^2)^{1-\frac{p}{2}}}\right]\d r+p\int_s^t\be\left[\frac{\langle Y_r^{s,y,x},f(r,x,0)\rangle}{(1+|Y_r^{s,y,x}|^2)^{1-\frac{p}{2}}}\right]\d r\\
        &\leq p\int_s^t\be\left[\frac{-\lambda|Y_r^{s,y,x}|^2}{(1+|Y_r^{s,y,x}|^2)^{1-\frac{p}{2}}}\right]\d r+p\int_s^t\be[|f(r,x,0)||Y_r^{s,y,x}|^{p-1}]\d r\\
        &\leq p\int_s^t\be\left[\frac{-\lambda|Y_r^{s,y,x}|^2}{(1+|Y_r^{s,y,x}|^2)^{1-\frac{p}{2}}}\right]\d r+p\varsigma \int_s^t\be[(1+|Y_r^{s,y,x}|^2)^\frac{p}{2}]\d r+C_p\varsigma^{1-p}\int_s^t|f(r,x,0)|^p\d r\\
        &\leq -p(\lambda-\varsigma)\int_s^t\be[(1+|Y_r^{s,y,x}|^2)^\frac{p}{2}]\d r+C_{\kappa,p}(\lambda+\varsigma^{1-p})(1+|x|^p)(t-s).
    \end{align*}
    To estimate $I_2$, we apply Young's inequality and \eqref{eq:bounded} to get
    \begin{align*}
        I_2&= p\int_s^t\int_{|z|>1}\int_0^1\be\left[\frac{\langle Y_r^{s,y,x}+ug(r,Y_r^{s,y,x})z,g(r,Y_r^{s,y,x})z\rangle}{(1+|Y_r^{s,y,x}+ug(r,Y_r^{s,y,x})z|^2)^{1-\frac{p}{2}}}\right]\d u\nu_2(\d z)\d r\\
        &\leq p\int_s^t\int_{|z|>1}\int_0^1\be[|Y_r^{s,y,x}+ug(r,Y_r^{s,y,x})z|^{p-1}|g(r,Y_r^{s,y,x})z|]\d u\nu_2(\d z)\d r\\
        &\leq p\int_s^t\int_{|z|>1}\be[|Y_r^{s,y,x}|^{p-1}|g(r,Y_r^{s,y,x})z|+|g(r,Y_r^{s,y,x})z|^p]\nu_2(\d z)\d r\\
        &\leq p\int_s^t\int_{|z|>1}\be[C_g^\alpha|Y_r^{s,y,x}|^p|z|+C_g^{\alpha(1-p)}C_p\Vert g(r,Y_r^{s,y,x})\Vert^p|z|+\Vert g(r,Y_r^{s,y,x})\Vert^p|z|^p]\nu_2(\d z)\d r\\
        &\leq pC_{\alpha,d_2}S_{d_2}(\alpha-1)^{-1}C_g^\alpha\int_s^t\be[(1+|Y_r^{s,y,x}|^2)^\frac{p}{2}]\d r+C_{\kappa,p}(t-s).
    \end{align*}
   By Lemma \ref{lem:H(y,z)} and \eqref{eq:bounded}, we have
    \begin{align*}
        I_3&\leq p\int_s^t\int_{|z|\leq 1}\int_0^1\int_0^1\be\left[\frac{\langle u_1g(r,Y_r^{s,y,x})z,g(r,Y_r^{s,y,x})z\rangle}{(1+|Y_r^{s,y,x}+u_2u_1g(r,Y_r^{s,y,x})z|^2)^{1-\frac{p}{2}}}\right]\d u_2\d u_1\nu_2(\d z)\d r \\
        &\leq p\int_s^t\int_{|z|\leq 1}\int_0^1 u_1\be[\Vert g(r,Y_r^{s,y,x})\Vert^2]|z|^2\d u_1\nu_2(\d z)\d r\\
        &\leq C_{\kappa,p}C_{\alpha,d_2}S_{d_2}(2-\alpha)^{-1}(t-s)
    \end{align*}
    Putting together the above estimates and taking $\varsigma=C_{\alpha,d_2}S_{d_2}(2-\alpha)^{-1}C_g^\alpha$, we derive
    \begin{align*}
        \be[V(Y_t^{s,y,x})]&\leq (1+|y|^2)^\frac{p}{2}-p(\lambda-q')\int_s^{t}\be[V(Y_r^{s,y,x})]\d r+C_{\kappa,p}(1+|x|^p)(t-s),
    \end{align*}
    where 
    \[q'=C_{\alpha,d_2}S_{d_2}((\alpha-1)^{-1}+(2-\alpha)^{-1})C_g^\alpha<q.\]
    Then we conclude that
    \begin{align*}
        \frac{\d}{\d t}\be[V(Y_t^{s,y,x})]\leq -p(\lambda-q)\be[V(Y_t^{s,y,x})]+C_{\kappa,p}(1+|x|^p).
    \end{align*}
    By the comparison theorem, it follows that
    \begin{equation*}
        \be[V(Y_t^{s,y,x})]\leq \re^{-p(\lambda-q)(t-s)}(1+|y|^2)^\frac{p}{2}+C_{\kappa,p}(1+|x|^p).
    \end{equation*}
Therefore, we get
\begin{align*}
    \be[|Y_t^{s,y,x}|^p]\leq C_{\kappa,p}(1+|x|^p+\re^{-p(\lambda-q)(t-s)}|y|^p).
\end{align*}
This concludes the proof.
\end{proof}

\begin{lemma}\label{lem:Y_t^{y_1,x_1}-Y_t^{y_2,x_2}}
    Suppose that Assumption \ref{cond} holds. For any $x_i\in\br^n$ and $y_i\in\br^m$, $i=1,2$, then there exists $C_{\kappa,p}>0$ such that for all $p\in[1,\alpha)$
    \begin{equation}
        \be[|Y_t^{s,y_1,x_1}-Y_t^{s,y_2,x_2}|^p]\leq \re^{-p(\lambda-q)(t-s)}|y_1-y_2|^p+C_{\kappa,p}|x_1-x_2|^p.
    \end{equation}
\end{lemma}

\begin{proof}
    We consider only the case of $p\in(1,\alpha)$, as the case of $p=1$ is similar. For any fixed $\theta>0$, we define $V(y)=(\theta+|y|^2)^\frac{p}{2}$. For notational simplicity, we set $Z_t=Y_t^{s,y_1,x_1}-Y_t^{s,y_2,x_2}$, $\tilde{f}_t=f(t,x_1,Y_t^{s,y_1,x_1})-f(t,x_2,Y_t^{s,y_2,x_2})$ and $\hat{g}_t:=g(t,Y_t^{s,y_1,x_1})-g(t,Y_t^{s,y_2,x_2})$. Applying It\^o's formula then gives
    \begin{align*}
    \be[V(Z_t)]&=(\theta+|y_1-y_2|^2)^\frac{p}{2}+\int_s^t\be[\langle \nabla V(Z_r),\tilde{f}_r\rangle]\d r\\
    &\quad+\int_s^t\int_{|z|>\frac{1}{2C_g}}\be[V(Z_r+\hat{g}_rz)-V(Z_r)]\nu_2(\d z)\d r\notag\\
    &\quad+\int_s^t\int_{|z|\leq \frac{1}{2C_g}}\be[V(Z_r+\hat{g}_rz)-V(Z_r)-\langle \nabla V(Z_r),\hat{g}_rz\rangle]\nu_2(\d z)\d r\\
    &=:I_1+I_2+I_3.
    \end{align*}
    For the term $I_1$, we derive from \eqref{dissipative condition}, \eqref{cond:Lip} and Young's inequality that
\begin{align*}
    I_1&= p\int_s^t\be[(\theta+|Z_r|^2)^{\frac{p}{2}-1}(\langle Z_r,\hat{f}_r\rangle+\langle Z_r,\bar{f}_r\rangle)]\d r\\
    &\leq p\int_s^t\be[-\lambda|Z_r|^2(\theta+|Z_r|^2)^{\frac{p}{2}-1}+|\bar{f}_r||Z_r|^{p-1}]\d r\\
    &\leq -p(\lambda-\varsigma)\int_s^t\be[V(Z_r)]\d r+p\lambda\theta^\frac{p}{2}(t-s)+C_p\varsigma^{1-p}\int_s^t\be[|\bar{f}_r|^p]\d r\\
    &\leq -p(\lambda-\varsigma)\int_s^t\be[V(Z_r)]\d r+C_{\kappa,p}(\lambda\theta^\frac{p}{2}+\varsigma^{1-p}|x_1-x_2|^p)(t-s),
\end{align*}
where $\hat{f}_t=f(t,x_1,Y_t^{s,y_1,x_1})-f(t,x_1,Y_t^{s,y_2,x_2})$ and $\bar{f}_t=f(t,x_1,Y_t^{s,y_2,x_2})-f(t,x_2,Y_t^{s,y_2,x_2})$. For the term $I_2$, by \eqref{cond:Lip} again, we have
\begin{align*}
    I_2&=p\int_s^t\int_{|z|>\frac{1}{2C_g}}\int_0^1\be[(\theta+|Z_r+u\hat{g}_rz|^2)^{\frac{p}{2}-1}\langle Z_r+u\hat{g}_rz,\hat{g}_rz\rangle]\d u\nu_2(\d z)\d r\\
    &\leq p\int_s^t\int_{|z|>\frac{1}{2C_g}}\int_0^1\be[|Z_r+u\hat{g}_rz|^{p-1}|\hat{g}_rz|]\d u\nu_2(\d z)\d r\\
    &\leq p\int_s^t\int_{|z|>\frac{1}{2C_g}}\be[|Z_r|^{p-1}|\hat{g}_rz|+|\hat{g}_rz|^p]\nu_2(\d z)\d r\\
    &\leq p\int_s^t\int_{|z|>\frac{1}{2C_g}}\be[|Z_r|^{p}](C_g|z|+C_g^p|z|^p)\nu_2(\d z)\d r\\
    &\leq pC_{\alpha,d_2}S_{d_2}(2^{\alpha-1}(\alpha-1)^{-1}+2^{\alpha-p}(\alpha-p)^{-1})C_g^\alpha\int_s^t\be[V(Z_r)]\d r.
\end{align*}
For the last term, we apply Lemma \ref{lem:H(y,z)} and \eqref{cond:Lip} to deduce
\begin{align*}
    I_3&\leq\int_s^t\int_{|z|\leq\frac{1}{2C_g}}\int_0^1\int_0^1\be\left[\frac{p\langle u_1\hat{g}_rz,\hat{g}_rz\rangle}{(\theta+|Z_r+u_2u_1\hat{g}_rz|^2)^{1-\frac{p}{2}}}\right]\d u_2\d u_1\nu_2(\d z)\d r\\
    &\leq pC_g^2\int_s^t\int_{|z|\leq\frac{1}{2C_g}}\be[| Z_r|^2(\theta+1/4|Z_r|^2)^{\frac{p}{2}-1}|z|^2\b1_{\{|Z_r|>0\}}]\nu_2(\d z)\d r\\
    &\leq 2^{2-p}pC_g^2\int_s^t\int_{|z|\leq\frac{1}{2C_g}}\be[V(Z_r)]|z|^2\nu_2(\d z)\d r\\
    &\leq 2^{\alpha-p}pC_{\alpha,d_2}S_{d_2}(2-\alpha)^{-1}C_g^\alpha\int_s^t\be[V(Z_r)]\d r.
\end{align*}
Combining the above estimates and taking $\varsigma=C_{\alpha,d_2}S_{d_2}(2^{\alpha-1}-2^{\alpha-p})(\alpha-p)^{-1}C_g^\alpha$, we obtain
\begin{align*}
    \be[V(Z_t)]
    \leq-p(\lambda-q')\int_s^t\be[V(Z_r)]\d r+\theta^\frac{p}{2}p\lambda+C_p|x_1-x_2|^p(t-s).
\end{align*}
where
\[q'=C_{\alpha,d_2}S_{d_2}(2^{\alpha-1}(\alpha-1)^{-1}+2^{\alpha-1}(\alpha-p)^{-1}+2^{\alpha-p}(2-\alpha)^{-1})C_g^\alpha<q.\]
Letting $\theta$ go to zero, we derive
\begin{align*}
    \frac{\d}{\d t}\be[|Y_t^{s,y_1,x_1}-Y_t^{s,y_2,x_2}|^p]&\leq -p(\lambda-q)\be[|Y_t^{s,y_1,x_1}-Y_t^{s,y_2,x_2}|^p]+C_{\kappa,p}|x_1-x_2|^p.
\end{align*}
Then the comparision theorem yields
\begin{align*}
    \be[|Y_t^{s,y_1,x_1}-Y_t^{s,y_2,x_2}|^p]\leq \re^{-p(\lambda-q)(t-s)}|y_1-y_2|^p+C_{\kappa,p}|x_1-x_2|^p.
\end{align*}
The proof is finished.
\end{proof}

The following theorem not only establishes the existence, uniqueness, and exponential ergodicity of
a periodic measure to the time-inhomogeneous frozen equation \eqref{frozen SDE}, but also provides moment estimates for the fast variable regarding the periodic measure.

\begin{theorem}\label{thm:periodic measures}
    Suppose that Assumptions \ref{cond} and \ref{con:periodic} hold. Then, the SDE \eqref{frozen SDE} admits a unique continuous periodic measure $\rho_t^{x}$ with period $\tau_2$, which satisfies 
    \begin{equation}\label{eq:rho_t convergence}
        \cw_1(P^x(s-n\tau_2,y,t,\cdot),\rho_t^x)\leq C(1+|x|+|y|)\re^{-(\lambda-q)n\tau_2}.
    \end{equation}
    Furthermore, for all $p\in[1,\alpha)$, there exists $C_{\kappa,p}>0$ such that 
    \begin{equation}
        \int_{\br^m}|w|^p\rho_t^{x}(\d w)\leq C_{\kappa,p}(1+|x|^p).
    \end{equation}
\end{theorem}

\begin{proof}
    We proceed by splitting the proof into the following four steps.
    \vskip5pt \noindent
    \textbf{Step 1.} For any $\rho,\bar{\rho}\in\cp(\br^m)$, we will show that 
    \begin{equation}\label{P*(s,s+ntau)rho leq rho}
        \cw_1(P^{*,x}(t-n\tau_2,t)\rho,P^{*,x}(t-n\tau_2,t)\bar{\rho})\leq \re^{-(\lambda-q)n\tau_2}\cw_1(\rho,\bar{\rho}).
    \end{equation}
    Let $\eta_1$ and $\eta_2$ be two absolute-integrable and $\cf_s$-measurable random variables such that
    \[\cw_1(\rho,\bar{\rho})=\be[|\eta_1-\eta_2|].\]
    Denote by $Y_t^{s,\eta_1,x}$ and $Y_t^{s,\eta_2,x}$ the solutions to \eqref{frozen SDE} with initial values $\eta_1$ and $\eta_2$, respectively. Then, Lemma \ref{lem:Y_t^{y_1,x_1}-Y_t^{y_2,x_2}} yields
    \begin{align*}
    \cw_1(P^{*,x}(t-n\tau_2,t)\rho,P^{*,x}(t-n\tau_2,t)\bar{\rho})&\leq \be[|Y_{t}^{t-n\tau_2,\eta_1,x}-Y_{t}^{t-n\tau_2,\eta_2,x}|]\\
    &\leq\re^{-(\lambda-q)n\tau_2}\be[|\eta_1-\eta_2|]\\
    &=\re^{-(\lambda-q)n\tau_2}\cw_1(\rho,\bar{\rho}).
\end{align*}
\textbf{Step 2.} Our goal is to establish the existence and uniqueness of the periodic measure. 
For all $k\geq n$, it is not difficult to see that
\begin{align*}
    \cw_1(P^x(t-n\tau_2,y,t,\cdot),P^x(t-k\tau_2,y,t,\cdot))&=\cw_1(P^{*,x}(t-n\tau_2,t)\delta_y,P^{*,x}(t-k\tau_2,t)\delta_y)\\
    &\leq \re^{-(\lambda-q)n\tau_2}\cw_1(\delta_y,P^{*,x}(t-k\tau_2,t-n\tau_2)\delta_y)\\
    &\leq \re^{-(\lambda-q)n\tau_2}\be[|Y_{t-n\tau_2}^{t-k\tau_2,y,x}-y|]\\
    &\leq C_\kappa(1+|x|+|y|)\re^{-(\lambda-q)n\tau_2}.
\end{align*}
Then we obtain
\[\lim_{n\to\infty}\cw_1(P^x(t-n\tau_2,y,t,\cdot),P^x(t-k\tau_2,y,t,\cdot))=0,\]
which implies that $\{P^x(t-n\tau_2,\eta,t,\cdot)\}_{n}$ is a Cauchy sequence and hence converges to a limit, denoted by
\[\rho_t^{y,x}:=\lim_{n\to\infty}P^x(t-n\tau_2,y,t,\cdot).\]
For all $y,\bar{y}\in\br^m$, we observe that
\begin{align*}
    \cw_1(\rho_t^{y,x},\rho_t^{\bar{y},x})&\leq\cw_1(\rho_t^{y,x},P^x(t-n\tau_2,y,t,\cdot))+\cw_1(P^x(t-n\tau_2,\bar{y},t,\cdot),\rho_t^{\bar{y},x})\\
    &\quad+\cw_1(P^x(t-n\tau_2,y,t,\cdot),P^x(t-n\tau_2,\bar{y},t,\cdot))
\end{align*}
and
\begin{align*}
    \cw_1(P^x(t-n\tau_2,y,t,\cdot),P^x(t-n\tau_2,\bar{y},t,\cdot))&\leq \cw_1(P^{*,x}(t-n\tau_2,t,\cdot)\delta_y,P^{*,x}(t-n\tau_2,t,\cdot)\delta_{\bar{y}})\\
    &\leq \re^{-(\lambda-q)n\tau_2}\be[|y-\bar{y}|].
\end{align*}
Therefore, $\cw_1(\rho_t^{y,x},\rho_t^{\bar{y},x})\leq 0$ as $n\to+\infty$, showing that the limit $\rho_t^{y,x}$ does not depend on $y$. It is denoted as $\rho_t^x$. Note that for all $t\geq s$, we have
\begin{align*}
    &~\cw_1(P^x(s-n\tau_2,y,t,\cdot),P^x(t-n\tau_2,y,t,\cdot))\\
    =&~\cw_1(P^{*,x}(t-n\tau_2,t)P^{*,x}(s-n\tau_2,t-n\tau_2)\delta_y,P^{*,x}(t-n\tau_2,t)\delta_y)\\
    \leq&~\re^{-(\lambda-q)n\tau_2}\cw_1(P^{*,x}(s,t)\delta_y,\delta_y)\\
    =&~\re^{-(\lambda-q)n\tau_2}\be[|Y_t^{s,y,x}-y|]\\
    \leq&~C_\kappa(1+|x|+|y|)\re^{-(\lambda-q)n\tau_2}.
\end{align*}
For all $n\geq k$, we get
\begin{align*}
    \cw_1(P^x(s-n\tau_2,y,t,\cdot),\rho_t^x)&\leq \cw_1(P^x(s-n\tau_2,y,t,\cdot),P^x(t-n\tau_2,y,t,\cdot))+\cw_1(P^x(t-k\tau_2,y,t,\cdot),\rho_t^x)\\
    &\quad+\cw_1(P^x(t-n\tau_2,y,t,\cdot),P^x(t-k\tau_2,y,t,\cdot))\\
    &\leq C(1+|x|+|y|)(\re^{-(\lambda-q)n\tau_2}+\re^{-(\lambda-q)k\tau_2})+\cw_1(P^x(t-k\tau_2,y,t,\cdot),\rho_t^x),
\end{align*}
then \eqref{eq:rho_t convergence} holds as $k\to+\infty$. 

We claim that $\rho_t^{x}$ is a periodic measure. To establish this, observe that for all $t\geq s$, we have
\begin{align*}
    \cw_1(\rho_{t+\tau_2}^x,\rho_t^x)&\leq\cw_1(\rho_{t+\tau_2}^x,P^x(t-(n+1)\tau_2,y,t-n\tau_2,\cdot))+\cw_1(P^x(t-n\tau_2,y,t,\cdot),\rho_t^x)\\
    &\quad+\cw_1(P^x(t-(n+1)\tau_2,y,t-n\tau_2,\cdot),P^x(t-n\tau_2,y,t,\cdot)),
\end{align*}
and
\begin{align*}
    \cw_1(P^{*,x}(s,t)\rho_s^{x},\rho_t^x)&\leq\cw_1(P^{*,x}(s,t)\rho_s^x,P^{*,x}(s,t)P^x(s-n\tau_2,y,s,\cdot))+\cw_1(P^x(s-n\tau_2,y,t,\cdot),\rho_t^x)\\
    &\quad+\cw_1(P^{*,x}(s,t)P^x(s-n\tau_2,y,s,\cdot),P^{x}(s-n\tau_2,y,t,\cdot)).
\end{align*}
Note that
\[P^{*,x}(s,t)P^x(s-n\tau_2,y,s,\cdot)=\int_{\br^m}P^x(s,w,t,\cdot)P^x(s-n\tau_2,y,s,\d w)=P^x(s-n\tau_2,y,t,\cdot),\]
then taking $n\to+\infty$, we conclude that
\[\cw_1(\rho_{s+\tau_2}^x,\rho_s^x)=0~~\mbox{and}~~\cw_1(P^{*,x}(s,t)\rho_s^{x},\rho_t^x)=0.\]
It remains to prove that $\rho_s^{x}$ is the unique periodic measure. Suppose that $\rho_s^{x}$ and $\bar{\rho}_s^{x}$ are two periodic measures, then we have
\begin{align*}
    \cw_1(\rho_s^{x},\bar{\rho}_s^{x})&=\cw_1(\rho_{s+\tau_2}^{x},\bar{\rho}_{s+\tau_2}^{x})\\
    &=\cw_1(P^{*,x}(s,s+\tau_2)\rho_{s}^{x},P^{*,x}(s,s+\tau_2)\bar{\rho}_{s}^{x})\\
    &\leq \re^{-(\lambda-q)\tau_2}\cw_1(\rho_{s}^{x},\bar{\rho}_{s}^{x}),
\end{align*}
which implies that $\rho_{s}^{x}=\bar{\rho}_{s}^{x}$.
\vskip5pt
\noindent\textbf{Step 3.} We will show that $\rho_s^x$ is continuous in $\cw_1$. For any $t\geq r\geq s$, note that
\begin{align*}
    Y_t^{s,y,x}-Y_r^{s,y,x}=\int_r^tf(u,x,Y_u^{s,y,x})\d u+\int_r^tg(u,x,Y_u^{s,y,x})\d L_u^2.
\end{align*}
Following the argument of Lemma \ref{lem:X(t+h)-X(t)}, we derive that
\begin{align}\label{eq:Yt-Yr}
    \be[|Y_t^{s,y,x}-Y_r^{s,y,x}|]\leq C_{\kappa}(1+|x|+|y|)(|t-r|+|t-r|^\frac{1}{\alpha}).
\end{align}
For any $t>r>s$, we have
\begin{equation}\label{eq:W(rhos,rhot)}
    \begin{aligned}
        \cw_1(\rho^x_r,\rho^x_t)&\leq\cw_1(\rho^x_r,P^x(s-n\tau_2,y,r,\cdot))+\cw_1(P^x(s-n\tau_2,y,t,\cdot),\rho^x_t)\\
        &~~~+\cw_1(P^x(s-n\tau_2,y,r,\cdot),P^x(s-n\tau_2,y,t,\cdot)).
    \end{aligned}
\end{equation}
It follows from the definition of the 1-Wasserstein distance that
    \begin{align*}
        \cw_1(P^x(s-n\tau_2,y,r,\cdot),P^x(s-n\tau_2,y,t,\cdot))&=\sup_{\vp\in Lip_1}|\be[\vp(Y_{r}^{s-n\tau_2,y,x})]-\be[\vp(Y_{t}^{s-n\tau_2,y,x})]|\\
        &\leq\be[|Y_r^{s-n\tau_2,y,x}-Y_t^{s-n\tau_2,y,x}|].
    \end{align*}
Together with \eqref{eq:Yt-Yr} and \eqref{eq:W(rhos,rhot)}, we first take the limit as $n\to+\infty$, and then derive that $\rho^x_r\to \rho^x_t$ in $\cw_1$ as $r\to t$.

\vskip5pt
\noindent\textbf{Step 4.} We will show that $\rho_s^x$ is in $\cm_p$. By Lemma \ref{lem:sup E[Y_t]}, we deduce that
\begin{align*}
    \int_{\br^m}|w|^p\rho_t^{x}(\d w)&=\lim_{n\to+\infty}\int_{\br^m}|w|^pP^x(s-n\tau_2,y,t,\d w)\\
    &\leq \limsup_{n\to+\infty}\be[|Y_{t}^{s-n\tau_2,y,x}|^p]\\
    &\leq C_{\kappa,p}\limsup_{n\to+\infty} (1+|x|^p+\re^{-p(\lambda-q)(t-s+n\tau_2)}|y|^p)\\
    &= C_{\kappa,p}(1+|x|^p).
\end{align*}
This completes the proof.
\end{proof}

We conclude this subsection by analysing the exponential convergence of the average coefficient, which is essential for proving the averaging principle.
\begin{lemma}\label{lem:bar E[b]-bar b}
    Suppose that Assumption \ref{cond} holds. For any $x\in\br^n$, $y\in\br^m$ and $s,t_1,t_2\in\br$, then there exists $C_\kappa>0$ such that
    \begin{align*}
            |\be[b(t_1,x,Y_{t_2}^{s,y,x})]-\bar{b}(t_1,t_2,x)|\leq C_\kappa(1+|x|+|y|)\re^{-(\lambda-q)(t_2-s)},
    \end{align*}
    where 
    \[\bar{b}(t_1,t_2,x)=\int_{\br^m}b(t_1,x,w)\rho_{t_2}^{x}(\d w).\]
\end{lemma}
\begin{proof}
    By the definition of periodic measures and Lemma \ref{lem:Y_t^{y_1,x_1}-Y_t^{y_2,x_2}}, we derive that
    \begin{align*}
        |\be[b(t_1,x,Y_{t_2}^{s,y,x})]-\bar{b}(t_1,t_2,x)|&=\left|\be[b(t_1,x,Y_{t_2}^{s,y,x})]-\int_{\br^m}b(t_1,x,w)\rho_{t_2}^{x}(\d w)\right|\\
        &=\left|\be[b(t_1,x,Y_{t_2}^{s,y,x})]-\int_{\br^m}b(t_1,x,w)P^{*,x}(s,t_2)\rho_s^{x}(\d w)\right|\\
        &=\left|\be[b(t_1,x,Y_{t_2}^{s,y,x})]-\int_{\br^m}\int_{\br^m}b(t_1,x,w)P^x(s,v,t_2,\d w)\rho_s^{x}(\d v)\right|\\
        &=\left|\be[b(t_1,x,Y_{t_2}^{s,y,x})]-\int_{\br^m}\be[b(t_1,x,Y_{t_2}^{s,v,x})]\rho_s^{x}(\d v)\right|\\
        &=\left|\int_{\br^m}\be[b(t_1,x,Y_{t_2}^{s,y,x})]-\be[b(t_1,x,Y_{t_2}^{s,v,x})]\rho_s^{x}(\d v)\right|\\
        &\leq C_\kappa\re^{-(\lambda-q)(t_2-s)}\int_{\br^m}\be[|y-v|]\rho_s^{x}(\d v)\\
        &\leq C_\kappa(1+|x|+|y|)\re^{-(\lambda-q)(t_2-s)}.
    \end{align*}
    This completes the proof.
\end{proof}

\subsection{The averaged equation of the first approximation}\label{subsec:The averaged equation}
In this subsection, we formulate the averaged equation with $\ve$ and provide a proof of the first main theorem. Specifically, recall that the averaged equation \eqref{thmeq:avraged equation1} takes the form
\begin{equation*}
\begin{cases}
\d \bar{X}_t^\ve=\bar{b}(t,t/\ve,\bar{X}_t^\ve)\d t+\sigma(t,\bar{X}_t^\ve) \d L_t^1,\\
\bar{X}_s^\ve=x,
\end{cases}
\end{equation*}
where
\[\bar{b}(t_1,t_2,x)=\int_{\br^m}b(t_1,x,y)\rho_{t_2}^{x}(\d y),\]
and $\rho_t^{x}$ denotes the unique periodic measure corresponding to the time-inhomogeneous frozen equation \eqref{frozen SDE}.

\begin{lemma}\label{lem:averaged eq has Lip and linear growth}
    Suppose that Assumption \ref{cond} holds. Then for all $x_i\in\br^n$ and $t_i\in\br$, $i=1,2$, there exists $C_\kappa>0$ such that
    \begin{align}\label{eq:averaged eq has Lip and linear growth}
        |\bar{b}(t_1,t_2,x_1)-\bar{b}(t_1,t_2,x_2)|\leq C_\kappa|x_1-x_2|,~~|\bar{b}(t_1,t_2,,x_1)|\leq C_\kappa(1+|x_1|).
    \end{align}
\end{lemma}

\begin{proof}
    It follows from Lemmas \ref{lem:bar E[b]-bar b} and \ref{lem:Y_t^{y_1,x_1}-Y_t^{y_2,x_2}}, and \eqref{cond:Lip} that
    \begin{align*}
        |\bar{b}(t_1,t_2,x_1)-\bar{b}(t_1,t_2,x_2)|&\leq |\bar{b}(t_1,t_2,x_1)-\be[b(t_1,x_1,Y_{t_2}^{s,y,x_1})]|\\
        &\quad+|\be[b(t_1,x_1,Y_{t_2}^{s,y,x_1})]-\be[b(t_1,x_2,Y_{t_2}^{s,y,x_2})]|\\
        &\quad+|\be[b(t_1,x_2,Y_{t_2}^{s,y,x_2})]-\bar{b}(t_1,t_2,x_2)|\\
        &\leq C_\kappa(1+|x_1|+|x_2|+|y|)\re^{-(\lambda-q)(t_2-s)}+C_\kappa|x_1-x_2|.
    \end{align*}
    Letting $s$ tend to $-\infty$, we derive that
    \begin{equation*}
    |\bar{b}(t_1,t_2,x_1)-\bar{b}(t_1,t_2,x_2)|\leq C_\kappa|x_1-x_2|.
\end{equation*}
Moreover, we have
\begin{align*}
    |\bar{b}(t_1,t_2,x_1)|&=\left|\int_{\br^m}b(t_1,x_1,y)\rho_{t_2}^{x_1}(\d y)\right|\leq C\int_{\br^m}(1+|x_1|+|y|)\rho_{t_2}^{x_1}(\d y)\leq C_\kappa(1+|x_1|).
\end{align*}
The proof is finished.
\end{proof}

\begin{lemma}\label{lem:bar{X}_t^ve}
    Suppose that Assumption \ref{cond} holds. Then for the initial condition $\bar{X}_s^\ve=x$, there exists a unique solution $\{\bar{X}_t^\ve, t\geq s\}$ to the system \eqref{thmeq:avraged equation1}. Moreover, there exists $C_{\kappa,p,s,T}>0$ such that for all $p\in[1,\alpha)$,
    \begin{align}\label{eq: E[|bar{X}_t|^p]}
        \sup_{\ve\in(0,1)}\sup_{t\in[s,T]}\be[\bar{X}_t^\ve|^p]\leq C_{\kappa,p,s,T}(1+|x|^p)
    \end{align}
    and for all $\Delta\in(0,1]$,
    \begin{align}\label{eq: E[|bar{X}_t^ve-bar{X}_{s+t(Delta)}|^p]}
        \sup_{\ve\in(0,1)}\sup_{t\in[s,T]}\be[|\bar{X}_t^\ve-\bar{X}_{t_\Delta}^\ve|^p]\leq C_{\kappa,p,s,T}(1+|x|^p)(\Delta^p+\Delta^\frac{p}{\alpha}).
    \end{align}
\end{lemma}

\begin{proof}
    Adapting the argument of \cite[Theorem 6.2.9]{A09}, we show that \eqref{thmeq:avraged equation1} has a unique solution. Furthermore, \eqref{eq: E[|bar{X}_t|^p]} and \eqref{eq: E[|bar{X}_t^ve-bar{X}_{s+t(Delta)}|^p]} follow arguments analogous to those for Lemmas \ref{lem:sup E[|X|] and sup E[|Y|]} and \ref{lem:X(t+h)-X(t)}, respectively.
\end{proof}

By combining the estimates obtained above,  we now give the proof of the first main result.

\begin{proof}[Proof of Theorem \ref{thm:X_t^ve-bar{X}_t^ve}]
    The proof of this theorem is divided into three steps as follows.\vskip5pt
    \noindent\textbf{Step 1.} For any $p\in(1,\alpha)$ and any fixed $\theta\in(0,1)$, we define $V(x)=(\theta+|x|^2)^\frac{p}{2}$. A simple calculation then yields
    \begin{align}\label{eq:|nabla V|}
        |\nabla V(x)|\leq p|x|^{p-1}.
    \end{align}
    According to \cite[Theorem 2.3]{LSWX25}, $\nabla V(x)$ is $(p-1)$-H\"oder continuous, i.e., for any $x_1,x_2\in\br^n$,
    \begin{equation}\label{eq: nabla V(x_1)-nabla V(x_2)}
        |\nabla V(x_2)-\nabla V(x_1)|\leq C_p|x_2-x_1|^{p-1}
    \end{equation}
    holds for some $C_p>0$. Note that
    \begin{align}
        \d(X_t^\ve-\bar{X}_t^\ve)=(b(t,X_t^\ve,Y_t^\ve)-\bar{b}(t,t/\ve,\bar{X}_t^\ve))\d t+(\sigma(t,X_t^\ve)-\sigma(t,\bar{X}_t^\ve))\d L_t^1.
    \end{align}
    Let $Z_t^\ve:=X_t^\ve-\bar{X}_t^\ve$, it is a consequence of It\^o's formula that
    \begin{align*}
        \be[V(Z_t^\ve)]&=\theta^\frac{p}{2}+\int_s^t\be[\langle\nabla V(Z_u^\ve),b(u,X_u^\ve,Y_u^\ve)-\bar{b}(u,u/\ve,\bar{X}_u^\ve)\rangle]\d u\\
        &\quad+\int_s^t\int_{|z|> \frac{1}{2C_\sigma}}\be[V(Z_u^\ve+\hat{\sigma}_uz)-V(Z_u^\ve)]\nu_1(\d z)\d u\\
        &\quad+\int_s^t\int_{|z|\leq \frac{1}{2C_\sigma}}\be[V(Z_u^\ve+\hat{\sigma}_uz)-V(Z_u^\ve)-\langle\nabla V(Z_u^\ve),\hat{\sigma}_uz\rangle]\nu_1(\d z)\d u,
    \end{align*}
    where $\hat{\sigma}_u=\sigma(u,X_u^\ve)-\sigma(u,\bar{X}_u^\ve)$. It follows from Lemma \ref{lem:averaged eq has Lip and linear growth} and \eqref{eq: nabla V(x_1)-nabla V(x_2)} that
    \begin{align*}
        \be[V(Z_t^\ve)]&\leq C_{\kappa,p}\theta^\frac{p}{2}+\int_s^t\be[\langle \nabla V(Z_u^\ve), b(u,X_u^\ve,Y_u^\ve)-\bar{b}(u,u/\ve,\bar{X}_u^\ve)\rangle]\d u+C_{\kappa,p}\int_s^t\be[V(Z_u^\ve)]\d u\\
        &=C_{\kappa,p}\theta^\frac{p}{2}+\int_s^t\be[\langle \nabla V(Z_u^\ve), b(u,X_u^\ve,Y_u^\ve)-\bar{b}(u,u/\ve,X_u^\ve)\rangle]\d u\\
        &\quad+\int_s^t\be[\langle \nabla V(Z_u^\ve), \bar{b}(u,u/\ve,X_u^\ve)-\bar{b}(u,u/\ve,\bar{X}_u^\ve)\rangle]\d u+C_{\kappa,p}\int_s^t\be[V(Z_u^\ve)]\d u \\
        &\leq C_\kappa\theta^\frac{p}{2}+C_{\kappa,p}\int_s^t\be[V(Z_u^\ve)]\d u+I,
    \end{align*}
    where
    \[I=\sup_{t\in[s,T]}\left|\int_s^t\be[\langle \nabla V(Z_u^\ve), b(u,X_u^\ve,Y_u^\ve)-\bar{b}(u,u/\ve,X_u^\ve)\rangle]\d u\right|.\]
    Therefore, Gronwall's inequality yields
    \begin{align}\label{eq: sup E[V(Z)] leq I1+I2+I3}
        \sup_{t\in[s,T]}\be[V(Z_t^\ve)]&\leq C_{\kappa,p,s,T}\theta^\frac{p}{2} +C_{\kappa,p,s,T}(I_1+I_2+I_3).
    \end{align}
    Here,
    \begin{align*}
        I_1&=\sup_{t\in[s,T]}\left|\int_s^t\be[\langle \nabla V(Z_u^\ve)-\nabla V(Z_{u_\Delta}^\ve), b(u,X_u^\ve,Y_u^\ve)-\bar{b}(u,u/\ve,X_u^\ve)\rangle]\d u\right|,\\
        I_2&=\sup_{t\in[s,T]}\left|\int_s^t\be[\langle \nabla V(Z_{u_\Delta}^\ve),\hat{b}_{u,\Delta}-\hat{\bar{b}}_{u,\Delta}\rangle]\d u\right|,\\
        I_3&=\sup_{t\in[s,T]}\left|\int_s^t\be[\langle \nabla V(Z_{u_\Delta}^\ve), b(u_\Delta,X_{u_\Delta}^\ve,\hat{Y}_u^\ve)-\bar{b}(u_\Delta,u/\ve,X_{u_\Delta}^\ve)\rangle]\d u\right|,
    \end{align*}
    where $\hat{b}_{u,\Delta}=b(u,X_u^\ve,Y_u^\ve)-b(u_\Delta,X_{u_\Delta}^\ve,\hat{Y}_u^\ve)$ and $\hat{\bar{b}}_{u,\Delta}=\bar{b}(u,u/\ve,X_u^\ve)-\bar{b}(u_\Delta,u/\ve,X_{u_\Delta}^\ve)$.
    \vskip5pt\noindent
    \textbf{Step 2.} This step is to estimate $I_1$, $I_2$ and $I_3$, respectively. For the term  $I_1$, we apply \eqref{eq: nabla V(x_1)-nabla V(x_2)}, \eqref{cond:linear growth}, \eqref{eq:averaged eq has Lip and linear growth}, \eqref{eq: E[|bar{X}_t^ve-bar{X}_{s+t(Delta)}|^p]}, H\"older's inequality, Lemmas \ref{lem:sup E[|X|] and sup E[|Y|]} and \ref{lem:X(t+h)-X(t)} to obtain
    \begin{equation}\label{estmate I1}
        \begin{aligned}
            I_1&\leq C_{p}\int_s^T\be[|b(u,X_u^\ve,Y_u^\ve)-\bar{b}(u,u/\ve,X_u^\ve)||Z_{u,\Delta}^\ve-\bar{Z}_{u,\Delta}^\ve|^{p-1}]\d u\\
            &\leq C_{p}\int_s^T(\be[|b(u,X_u^\ve,Y_u^\ve)-\bar{b}(u,u/\ve,X_u^\ve)|^p])^\frac{1}{p}(\be[|Z_{u,\Delta}^\ve-\bar{Z}_{u,\Delta}^\ve|^p])^\frac{p-1}{p}\d u\\
            &\leq C_{\kappa,p,s,T}\int_s^T(\be[1+|X_u^\ve|^p+|Y_u^\ve|^p])^\frac{1}{p}((\be[|Z_{u,\Delta}^\ve|^p])^\frac{p-1}{p}+(\be[|\bar{Z}_{u,\Delta}^\ve|^p])^\frac{p-1}{p})\d u\\
            &\leq C_{\kappa,p,s,T}(1+|x|^p+|y|^p)(\Delta^{p-1}+\Delta^\frac{p-1}{\alpha}),
        \end{aligned}
    \end{equation}
    where $Z_{u,\Delta}^\ve=X_u^\ve-X_{u_\Delta}^\ve$ and $\bar{Z}_{u,\Delta}^\ve=\bar{X}_u^\ve-\bar{X}_{u_\Delta}^\ve$.  For the term $I_2$, it follows from \eqref{cond:Lip}, \eqref{eq:averaged eq has Lip and linear growth}, \eqref{eq:|nabla V|}, H\"older's inequality, Lemmas \ref{lem:sup E[|X|] and sup E[|Y|]}, \ref{lem:Y(t)-hatY(t)}, \ref{lem:X(t+h)-X(t)} and \ref{lem:bar{X}_t^ve} that
    \begin{equation}\label{estmate I2}
        \begin{aligned}
            I_2&\leq C_{p}\int_s^T\be[|Z_{u_\Delta}^\ve|^{p-1}(|u-u_\Delta|^\frac{1}{\alpha}+|Z_{u,\Delta}^\ve|+|Y_u^\ve-\hat{Y}_u^\ve|)]\d u\\
            &\leq C_{\kappa,p}\int_s^T(\be[|Z_{u_\Delta}^\ve|^p])^\frac{p-1}{p}(\Delta^\frac{p}{\alpha}+\be[|Z_{u,\Delta}^\ve|^p+|Y_u^\ve-\hat{Y}_u^\ve|^p])^\frac{1}{p}\d u\\
            &\leq C_{\kappa,p}\int_s^T((\be[|X_{u_\Delta}^\ve|^p])^\frac{p-1}{p}+(\be[|\bar{X}_{u_\Delta}^\ve|^p])^\frac{p-1}{p})(\Delta^\frac{p}{\alpha}+\be[|Z_{u,\Delta}^\ve|^p+|Y_u^\ve-\hat{Y}_u^\ve|^p])^\frac{1}{p}\d u\\
            &\leq C_{\kappa,p,s,T}(1+|x|^{p}+|y|^{p})(\Delta+\Delta^\frac{1}{\alpha}).
        \end{aligned}
    \end{equation}
    Let $t(\Delta)=\lfloor(t-s)/\Delta\rfloor$. For the term $I_3$, it is not difficult to see that
    \begin{equation}\label{estmate I3}
        \begin{aligned}
            I_3&\leq \sup_{t\in[s,T]}\left|\sum_{k=0}^{t(\Delta)-1}\int_{s+k\Delta}^{s+(k+1)\Delta}\be[\langle \nabla V(Z_{u_\Delta}^\ve),b(u_\Delta,X_{u_\Delta}^\ve,\hat{Y}_u^\ve)-\bar{b}(u_\Delta,u/\ve,X_{u_\Delta}^\ve)\rangle]\d u\right|\\
            &\quad+\sup_{t\in[s,T]}\left|\int_{t_\Delta}^t\be[\langle \nabla V(Z_{u_\Delta}^\ve),b(u_\Delta,X_{u_\Delta}^\ve,\hat{Y}_u^\ve)-\bar{b}(u_\Delta,u/\ve,X_{u_\Delta}^\ve)\rangle]\d u\right|\\
            &=:I_{31}+I_{32}.
        \end{aligned}
    \end{equation}
    We now turn to the estimation of the term $I_{31}$. For any random variables $\xi,\eta\in\cf_s$, consider the integral equation
    \begin{align*}
        \tilde{Y}_t^{\ve,s,\eta,\xi}=\eta+\ve^{-1}\int_s^tf(u/\ve,\xi,\tilde{Y}_u^{\ve,s,\eta,\xi})\d u+\ve^{-\frac{1}{\alpha}}\int_s^tg(u/\ve,\tilde{Y}_u^{\ve,s,\eta,\xi})\d L_u^2.
    \end{align*}
    For all $k\in\bn$ and $s+k\Delta\leq t< s+(k+1)\Delta$, it follows from the definition of $\hat{Y}_t^\ve$ that 
    \begin{align*}
        \hat{Y}_t^\ve=\tilde{Y}_t^{\ve,s+k\Delta,\hat{Y}_{s+k\Delta}^\ve,X_{s+k\Delta}^\ve}.
    \end{align*}
   Then for any $s+k\Delta\leq u\leq s+(k+1)\Delta$, we derive that
    \begin{align*}
        &~\be[\langle \nabla V(Z_{s+k\Delta}^\ve),b(s+k\Delta,X_{s+k\Delta}^\ve,\hat{Y}_u^\ve)-\bar{b}(s+k\Delta,u/\ve,X_{s+k\Delta}^\ve)\rangle]\\
        =&~\be[\be[\langle\nabla V(Z_{s+k\Delta}^\ve),b(s+k\Delta,X_{s+k\Delta}^\ve,\Tilde{Y}_u^{\ve,s+k\Delta,\hat{Y}_{s+k\Delta}^\ve,X_{s+k\Delta}^\ve})-\bar{b}(s+k\Delta,u/\ve,X_{s+k\Delta}^\ve)\rangle|\cf_{s+k\Delta}]]\\
        =&~\be[\langle\nabla V(z),\be[b(s+k\Delta,x,\Tilde{Y}_u^{\ve,s+k\Delta,y,x})-\bar{b}(s+k\Delta,u/\ve,x)]\rangle\b1_{y=\hat{Y}_{s+k\Delta}^\ve,x=X_{s+k\Delta}^\ve,z=Z_{s+k\Delta}^\ve}],
    \end{align*}
    where the last equality follows from the independence of $\tilde{Y}_t^{\ve,s+k\Delta,y,x}$ from $\cf_{s+k\Delta}$, and the $\cf_{s+k\Delta}$-measurability of $X_{s+k\Delta}^\ve$, $Y_{s+k\Delta}^\ve$ and $X_{s+k\Delta}^\ve-\bar{X}_{s+k\Delta}^\ve$. 
    By the definition of $\Tilde{Y}_t^{\ve,s,y,x}$, we have
    \begin{equation}\label{eq:tilde Y_{ve t}^{ve,s+k Delta,x,y}}
    \begin{aligned}
        \tilde{Y}_{\ve t}^{\ve,s+k\Delta,y,x}&=y+\ve^{-1}\int_{s+k\Delta}^{\ve t} f(u/\ve,x,\tilde{Y}_u^{\ve,s+k\Delta,y,x})\d u+\ve^{-\frac{1}{\alpha}}\int_{s+k\Delta}^{\ve t}g(u/\ve,\tilde{Y}_u^{\ve,s+k\Delta,y,x})\d L_u^2\\
        &=y+\int_{(s+k\Delta)/\ve}^tf(u,x,\tilde{Y}_{\ve u}^{\ve,s+k\Delta,y,x})\d u+\int_{(s+k\Delta)/\ve}^tg(u,\tilde{Y}_{\ve u}^{\ve,s+k\Delta,y,x})\d \bar{L}_u^2,
    \end{aligned}
    \end{equation}
    for all $s+k\Delta\leq \ve t< s+(k+1)\Delta$, where $\bar{L}_u^2=\ve^{-\frac{1}{\alpha}}L_{\ve u}^2$. Recall that
    \begin{equation}\label{eq:Y_t^{(s+k Delta) ve,x,y}}
        Y_t^{(s+k\Delta)/\ve,y,x}=y+\int_{(s+k\Delta)/\ve}^tf(u,x,Y_u^{(s+k\Delta)/\ve,y,x})\d u+\int_{(s+k\Delta)/\ve}^tg(u,Y_u^{(s+k\Delta)/\ve,y,x})\d L_u^2.
    \end{equation}
    By the uniqueness of the solutions to \eqref{eq:tilde Y_{ve t}^{ve,s+k Delta,x,y}} and \eqref{eq:Y_t^{(s+k Delta) ve,x,y}}, the processes  $\tilde{Y}_t^{\ve,s+k\Delta,y,x}$ and $Y_{t/\ve}^{(s+k\Delta)/\ve,y,x}$ have the same distribution on the interval $[s+k\Delta, s+(k+1)\Delta)$. 
    It follows from the Markov property of $Y_t^{s,x,y}$ and Lemma \ref{lem:bar E[b]-bar b} that
        \begin{align}\label{est I31}
            I_{31}&\leq \sum_{k=0}^{T(\Delta)-1}\left|\int_{s+k\Delta}^{s+(k+1)\Delta}\be[\langle \nabla V(z),\be[b(s+k\Delta,x,Y_{u/\ve}^{(s+k\Delta)/\ve,y,x})-\bar{b}(s+k\Delta,u/\ve,x)]\rangle\right.\notag\\
            &\left.\qquad\qquad\qquad\qquad\qquad\qquad\qquad\qquad\qquad\qquad\quad\times\b1_{\{y=\hat{Y}_{s+k\Delta}^\ve,x=X_{s+k\Delta}^\ve,z=Z_{s+k\Delta}^\ve\}}]\d u\right|\notag\\
            &\leq C_{\kappa,p}\sum_{k=0}^{T(\Delta)-1}\int_{s+k\Delta}^{s+(k+1)\Delta}\be[|Z_{s+k\Delta}^\ve|^{p-1}(1+|X_{s+k\Delta}^\ve|+|\hat{Y}_{s+k\Delta}^\ve|)]\re^{-(\lambda-q)(u-s-k\Delta)/\ve}\d u\notag\\
            &\leq C_{\kappa,p}\sum_{k=0}^{T(\Delta)-1}\int_{s+k\Delta}^{s+(k+1)\Delta}(\be[|Z_{s+k\Delta}^\ve|^p])^\frac{p-1}{p}(\be[1+|X_{s+k\Delta}^\ve|^p+|\hat{Y}_{s+k\Delta}^\ve|^p])^\frac{1}{p}\re^{-(\lambda-q)(u-s-k\Delta)/\ve}\d u\notag\\
            &\leq C_{\kappa,p,s,T}(1+|x|^p+|y|^p)\Delta^{-1}\max_{k\in[0,T(\Delta)-1]}\int_{s+k\Delta}^{s+(k+1)\Delta}\re^{-(\lambda-q)(u-s-k\Delta)/\ve}\d u\notag\\
            &\leq C_{\kappa,p,s,T}(1+|x|^p+|y|^p)\ve\Delta^{-1}.
        \end{align}
    For the term $I_{32}$, we use H\"older's inequality and get
    \begin{equation}\label{est I32}
        \begin{aligned}
            I_{32}&\leq  C_{\kappa,p}\sup_{u\in[s,T]}\be[(1+|X_{u_\Delta}^\ve|+|\hat{Y}_u^\ve|)|Z_{u_\Delta}^\ve|^{p-1}]\Delta\\
            &\leq  C_{\kappa,p}\sup_{u\in[s,T]}(\be[(1+|X_{u_\Delta}^\ve|^p+|\hat{Y}_u^\ve|^{p}])^\frac{1}{p}(\be[|Z_{u_\Delta}^\ve|^{p}])^\frac{p-1}{p}\Delta\\
            &\leq C_{\kappa,p,s,T}(1+|x|^p+|y|^p)\Delta. 
        \end{aligned}
    \end{equation}
    Substituting \eqref{est I31} and \eqref{est I32} into \eqref{estmate I3}, we have
    \begin{align}\label{estmate I3n}
        I_3\leq C_{\kappa,p,s,T}(1+|x|^p+|y|^p)(\ve\Delta^{-1}+\Delta).
    \end{align}
    \vskip5pt
    \noindent\textbf{Step 3.} From a combination of \eqref{estmate I1}, \eqref{estmate I2} and \eqref{estmate I3n}, we derive that
    \begin{align}\label{eq:step3}
        \sup_{t\in[s,T]}\be[(\theta+|X_t^\ve-\bar{X}_t^\ve|^2)^\frac{p}{2}]
        &\leq C_{\kappa,p,s,T}(1+|x|^p+|y|^p)(\Delta^\frac{p-1}{\alpha}+\ve\Delta^{-1}).
    \end{align}
    Choosing $\Delta=\ve^{\frac{\alpha}{\alpha+p-1}}$ and letting $\theta$ tend to $0$, we conclude that
    \begin{equation*}\label{eq:|X^ve-bar X^ve|^p leq ve}
        \sup_{t\in[s,T]}\be[|X_t^\ve-\bar{X}_t^\ve|^p]\leq C_{\kappa,p,s,T}(1+|x|^p+|y|^p)\ve^\frac{p-1}{\alpha+p-1}.
    \end{equation*}
    This completes the proof.
\end{proof}

 \subsection{The averaged equation of the second approximation}\label{subsec:The averaged equation without ve}
 This section is devoted to presenting the construction of the averaged equation independent of $\ve$ and proving the second main theorem. Recall that the averaged equation \eqref{eq:SDEAV2} is given by
 \begin{equation*}
        \begin{cases}
          \d \bar{X}_t=\bar{b}(t,\bar{X}_t)\d t+\sigma(t,\bar{X}_t)\d L_t^1,\\
          \bar{X}_s=x,
        \end{cases}
     \end{equation*}
   where
   \[\bar{b}(t_1,x)=\frac{1}{\tau_2}\int_0^{\tau_2}\bar{b}(t_1,t_2,x)\d t_2.\]
   By Lemma \ref{lem:averaged eq has Lip and linear growth}, a straightforward calculation yields that for any $t\in\br$ and $x_1,x_2\in\br^m$, 
   \begin{align}\label{eq:bar b Lip}
       |\bar{b}(t,x_1)-\bar{b}(t,x_2)|\leq C_\kappa|x_1-x_2|,~~\bar{b}(t,x_1)|\leq C_\kappa(1+|x_1|),
   \end{align}
   and for any $t,s\in\br$ and $x\in\br^m$,
   \begin{equation}\label{eq:bar b Holder}
       \begin{aligned}
           |\bar{b}(t,x)-\bar{b}(s,x)|\leq |t-s|^\frac{1}{\alpha}.
       \end{aligned}
   \end{equation}
   Consequently, it follows from \cite[Theorem 6.2.9]{A09} that SDE \eqref{eq:SDEAV2} admits a unique solution.

   \begin{lemma}\label{lem:bar X}
       Suppose that Assumptions \ref{cond} and \ref{con:periodic} hold. Then there exists $C_{\kappa,p,s,T}>0$ such that for all $p\in[1,\alpha)$,
       \begin{equation}
        \sup_{t\in[s,T]}\be[|\bar{X}_t|^p]\leq C_{\kappa,p,s,T}(1+|x|^p),
    \end{equation}
    and for all $\Delta\in(0,1)$,
    \begin{align}\label{eq: E[|bar{X}_t-bar{X}_{s+t(Delta)}|^p]}
        \sup_{t\in[s,T]}\be[|\bar{X}_t-\bar{X}_{t_\Delta}|^p]\leq C_{\kappa,p,s,T}(1+|x|^p)(\Delta^p+\Delta^\frac{p}{\alpha}).
    \end{align}
   \end{lemma}

   \begin{proof}
       The proof is omitted here due to its similarity to that of Lemma \ref{lem:bar{X}_t^ve}.
   \end{proof}

   \begin{lemma}\label{lem:averaged linear growth}
       For any $x\in\br^n$ and $T,t>0$, there exists $C>0$ such that
       \begin{equation*}
           \left|\frac{1}{T}\int_t^{t+T}\bar{b}(r_1,r_2,x)\d r_2-\frac{1}{\tau_2}\int_0^{\tau_2}\bar{b}(r_1,r_2,x)\d r_2\right|\leq \frac{C\tau_2}{T}(1+|x|).
       \end{equation*}
   \end{lemma}

   \begin{proof}
       For any $T>0$, there exists a nonnegative integer $M$ such that $M\tau_2\leq T< (M+1)\tau_2$. Note that
       \begin{align*}
           \frac{1}{T}\int_t^{t+T}\bar{b}(r_1,r_2,x)\d r_2=\frac{1}{T}\int_t^{t+M\tau_2}\bar{b}(r_1,r_2,x)\d r_2+\frac{1}{T}\int_{t+M\tau_2}^{t+T}\bar{b}(r_1,r_2,x)\d r_2.
       \end{align*}
       Since $\bar{b}(t_1,t_2,x)$ is the periodic function in $t_2$ with the period $\tau_2$, it follows that
       \begin{align*}
           \int_t^{t+\tau_2}\bar{b}(r_1,r_2,x)\d r_2=\int_0^{\tau_2}\bar{b}(r_1,r_2,x)\d r_2.
       \end{align*}
       Therefore, we derive that
       \begin{align*}
           &~\left|\frac{1}{T}\int_t^{t+T}\bar{b}(r_1,r_2,x)\d r_2-\frac{1}{\tau_2}\int_0^{\tau_2}\bar{b}(r_1,r_2,x)\d r_2\right|\\
           =&~\left|\frac{M}{T}\int_0^{\tau_2}\bar{b}(r_1,r_2,x)\d r_2-\frac{1}{\tau_2}\int_0^{\tau_2}\bar{b}(r_1,r_2,x)\d r_2+\frac{1}{T}\int_{t+M\tau_2}^{t+T}\bar{b}(r_1,r_2,x)\d r_2\right|\\
           \leq&~C(1+|x|)\left(\left|\frac{M}{T}-\frac{1}{\tau_2}\right|\tau_2+\frac{\tau_2}{T}\right)\\
           \leq&~\frac{C\tau_2}{T}(1+|x|),
       \end{align*}
       where the penultimate inequality is obtained from the linear growth property of $\bar{b}(t_1,t_2,x)$.
   \end{proof}
   
   Building on the above discussion, we now turn to proving the second averaging principle.
   \begin{proof}[Proof of Theorem \ref{thm:X_t^ve-bar X}]
    Notice that
    \begin{align*}
        \d(\bar{X}_t^\ve-\bar{X}_t)&=(\bar{b}(t,t/\ve,\bar{X}_t^\ve)-\bar{b}(t,\bar{X}_t))\d t+(\sigma(t,\bar{X}_t^\ve)-\sigma(t,\bar{X}_t))\d L_t^1.
    \end{align*}
    For any $p\in(1,\alpha)$ and any fixed $\theta\in(0,1)$, we define $V(x)=(\theta+|x|^2)^\frac{p}{2}$. Denote $\bar{Z}_t^\ve:=\bar{X}_t^\ve-\bar{X}_t$, it is a consequence of It\^o's formula that
    \begin{equation}\label{eq:nabla d X}
        \begin{aligned}
            \be[V(\bar{Z}_t^\ve)]&=\theta^\frac{p}{2}+\int_s^t\be[\langle\nabla V(\bar{Z}_u^\ve),\bar{b}(u,u/\ve,\bar{X}_u^\ve)-\bar{b}(u,\bar{X}_u)\rangle]\d u\\
            &\quad+\int_s^t\int_{|z|> \frac{1}{2C_\sigma}}\be[V(\bar{Z}_u^\ve+\hat{\sigma}_uz)-V(\bar{Z}_u^\ve)]\nu_1(\d z)\d u\\
            &\quad+\int_s^t\int_{|z|\leq \frac{1}{2C_\sigma}}\be[V(\bar{Z}_u^\ve+\hat{\sigma}_uz)-V(\bar{Z}_u^\ve)-\langle\nabla V(\bar{Z}_u^\ve),\hat{\sigma}_uz\rangle]\nu_1(\d z)\d u,
        \end{aligned}
    \end{equation}
    where $\hat{\sigma}_u:=\sigma(u,\bar{X}_u^\ve)-\sigma(u,\bar{X}_u)$. By following arguments analogous to those in \eqref{eq: sup E[V(Z)] leq I1+I2+I3}, we obtain
    \begin{align*}
        \sup_{t\in[s,T]}\be[V(\bar{Z}_t^\ve)]&\leq C_{\kappa,p,s,T}\theta^\frac{p}{2} +C_{\kappa,p,s,T}(I_1+I_2+I_3),
    \end{align*}
    where
    \begin{align*}
        I_1&=\sup_{t\in[s,T]}\left|\int_s^t\be[\langle \nabla V(\bar{Z}_u^\ve)-\nabla V(\bar{Z}_{u_\Delta}^\ve), \bar{b}(u,u/\ve,\bar{X}_u^\ve)-\bar{b}(u,\bar{X}_u^\ve)\rangle]\d u\right|,\\
        I_2&=\sup_{t\in[s,T]}\left|\int_s^t\be[\langle \nabla V(\bar{Z}_{u_\Delta}^\ve),(\bar{b}(u,u/\ve,\bar{X}_u^\ve)-\bar{b}(u_\Delta,u/\ve,\bar{X}_{u_\Delta}^\ve))-(\bar{b}(u,\bar{X}_u^\ve)-\bar{b}(u_\Delta,\bar{X}_{u_\Delta}^\ve))\rangle]\d u\right|,\\
        I_3&=\sup_{t\in[s,T]}\left|\int_s^t\be[\langle \nabla V(\bar{Z}_{u_\Delta}^\ve), \bar{b}(u_\Delta,u/\ve,\bar{X}_{u_\Delta}^\ve)-\bar{b}(u_\Delta,\bar{X}_{u_\Delta}^\ve)\rangle]\d u\right|.
    \end{align*}
    We now turn to estimate $I_1$, it follows from \eqref{eq: nabla V(x_1)-nabla V(x_2)}, H\"older's inequality, \eqref{cond:linear growth}, Lemmas \ref{lem:bar{X}_t^ve} and \ref{lem:bar X} that
    \begin{equation}\label{eq:2I1}
    \begin{aligned}
        I_1&\leq C_p\int_s^T\be[|\bar{Z}_u^\ve-\bar{Z}_{u_\Delta}|^{p-1}|\bar{b}(u,u/\ve,\bar{X}_u^\ve)-\bar{b}(u,\bar{X}_u^\ve)|]\d u\\
        &\leq C_p\int_s^T\be[|\bar{Z}_{u,\Delta}^\ve-\bar{Z}_{u,\Delta}|^{p-1}(1+|\bar{X}_u^\ve|)]\d u\\
        &\leq C_p\int_s^T((\be[|\bar{Z}_{u,\Delta}^\ve|^p])^\frac{p-1}{p}+\be[|\bar{Z}_{u,\Delta}|^p])^\frac{p-1}{p})(1+\be[|\bar{X}_u^\ve|^p])^\frac{1}{p}\d u\\
        &\leq C_{\kappa,p,s,T}(1+|x|^p)(\Delta^{p-1}+\Delta^\frac{p-1}{\alpha}),
    \end{aligned}
    \end{equation}
    where $\bar{Z}_{u,\Delta}^\ve=\bar{X}_u^\ve-\bar{X}_{u_\Delta}^\ve$ and $\bar{Z}_{u,\Delta}=\bar{X}_u-\bar{X}_{u_\Delta}$. For the term $I_2$, using \eqref{eq:bar b Holder}, and applying the H\"older's inequality, together with Lemmas \ref{lem:bar{X}_t^ve} and \ref{lem:bar X} again, we obtain
    \begin{equation}\label{eq:2I2}
    \begin{aligned}
        I_2&\leq C_p\int_s^T\be[|\bar{Z}_{u_\Delta}^\ve|^{p-1}(|u-u_\Delta|^\frac{1}{\alpha}+|\bar{Z}_{u,\Delta}^\ve|)]\d u\\
        &\leq C_\kappa\int_s^T(\be[|\bar{Z}_{u_\Delta}^\ve|^p])^\frac{p-1}{p}(\Delta^\frac{p}{\alpha}+\be[|\bar{Z}_{u,\Delta}^\ve|^p])^\frac{1}{p}\d u\\
        &\leq C_{\kappa,p}\int_s^T((\be[|\bar{X}_{u_\Delta}^\ve|^p])^\frac{p-1}{p}+(\be[|\bar{X}_{u_\Delta}|^p])^\frac{p-1}{p})(\Delta^\frac{p}{\alpha}+\be[|\bar{Z}_{u,\Delta}^\ve|^p])^\frac{1}{p}\d u\\
        &\leq C_{\kappa,p,s,T}(1+|x|^p)(\Delta+\Delta^\frac{1}{\alpha}).
    \end{aligned}
    \end{equation}
    Recall that $t_\Delta=s+\lfloor(t-s)/\Delta\rfloor\Delta$ and $t(\Delta)=\lfloor(t-s)/\Delta\rfloor$. For the term $I_3$,  we derive that
    \begin{align*}
        I_3&\leq\sup_{t\in[s,T]}\left|\sum_{k=0}^{t(\Delta)-1}\be\left[\int_{s+k\Delta}^{s+(k+1)\Delta}\langle\nabla V(\bar{Z}_{s+k\Delta}^\ve),\bar{B}(u/\ve)\rangle\d u\right]\right|\\
        &\quad+\sup_{t\in[s,T]}\left|\int_{t_\Delta}^t\be[\nabla V(\bar{Z}_{u_\Delta}^\ve),\bar{b}(u_\Delta,u/\ve,\bar{X}_{u_\Delta}^\ve)-\bar{b}(u_\Delta,\bar{X}_{u_\Delta}^\ve)\rangle]\d u\right|\\
        &=:I_{31}+I_{32},
    \end{align*}
    where $\bar{B}(u)=\bar{b}(s+k\Delta,u,\bar{X}_{s+k\Delta}^\ve)-\bar{b}(s+k\Delta,\bar{X}_{s+k\Delta}^\ve)$. We estimate $I_{31}$ and $I_{32}$ separately as follows. For the term $I_{31}$, we have
    \begin{align*}
        I_{31}&\leq \sum_{k=0}^{T(\Delta)-1}\be\left[\left|\int_{s+k\Delta}^{s+(k+1)\Delta}\langle\nabla V(\bar{Z}_{s+k\Delta}^\ve),\bar{B}(u/\ve)\rangle\d u\right|\right]\\
        &\leq \frac{C_{p,s,T}}{\Delta}\max_{k\in[0,T(\Delta)-1]}\be\left[|\bar{Z}_{s+k\Delta}^\ve|^{p-1}\left|\int_{s+k\Delta}^{s+(k+1)\Delta}\bar{B}(u/\ve)\d u\right|\right].
    \end{align*}
    Note that
    \begin{align*}
        &\int_{s+k\Delta}^{s+(k+1)\Delta}\bar{B}(u/\ve)\d u=\ve\int_{\frac{s+k\Delta}{\ve}}^{\frac{s+k\Delta}{\ve}+\frac{\Delta}{\ve}}\bar{B}(u)\d u,
    \end{align*}
    then it follows from Lemma \ref{lem:averaged linear growth} that
    \begin{align*}
        I_{31}
        &\leq C_{p,s,T}\max_{k\in[0,T(\Delta)-1]}\be\left[|\bar{Z}_{s+k\Delta}^\ve|^{p-1}\left|\frac{\ve}{\Delta}\int_{\frac{s+k\Delta}{\ve}}^{\frac{s+k\Delta}{\ve}+\frac{\Delta}{\ve}}\bar{B}(u)\d u\right|\right]\\
        &\leq C_{p,s,T}\max_{k\in[0,T(\Delta)-1]}\left(\be\left[\left|\frac{\ve}{\Delta}\int_{\frac{s+k\Delta}{\ve}}^{\frac{s+k\Delta}{\ve}+\frac{\Delta}{\ve}}\bar{B}(u)\d u\right|^p\right]\right)^\frac{1}{p}(\be[|\bar{Z}_{s+k\Delta}^\ve|^p])^\frac{p-1}{p}\\
        &\leq C_{p,s,T}\sup_{t\in[s,T]}(1+\be[|\bar{X}_t^\ve|^{p}]+\be[|\bar{X}_t|^{p}])\ve\Delta^{-1}\\
        &\leq C_{\kappa,p,s,T}(1+|x|^p)\ve\Delta^{-1}.
    \end{align*}
    For the term $I_{32}$, we have
    \begin{align*}
        I_{32}&\leq  C_{\kappa,p,s,T}\sup_{u\in[s,T]}\be[|Z_{u_\Delta}^\ve|^{p-1}(1+|\bar{X}_{u_\Delta}^\ve|)]\Delta\\
            &\leq  C_{\kappa,p,s,T}\sup_{u\in[s,T]}(\be[|Z_{u_\Delta}^\ve|^{p}])^\frac{p-1}{p}(1+\be[|\bar{X}_{u_\Delta}^\ve|^p])^\frac{1}{p}\Delta\\
            &\leq C_{\kappa,p,s,T}(1+|x|^p)\Delta.
    \end{align*}
    Therefore, we derive that
    \begin{equation}\label{eq:2I3}
        I_3\leq C_{\kappa,p,s,T}(1+|x|^p)(\ve\Delta^{-1}+\Delta).
    \end{equation}
    The rest of the proof of this theorem can be completed similarly as that of Theorem \ref{thm:X_t^ve-bar{X}_t^ve}, together with an application of the triangle inequality.
\end{proof}

\section{Application to the climate-weather multiscale system with $\alpha$-stable noise}\label{sec:Example}
In this section, we present an example of a real physical system to illustrate the theoretical results derived in the previous sections. For simplicity, the discussion is confined to the one dimensional setting. Nevertheless, the extension to higher dimensions is straightforward.

The individual dynamics of the weather and climate systems driven by continuous noise has been discussed in \cite{BB07,BPSV83}. The connection of these weather and climate models with the multiscale system can be seen in Appendix \ref{A}, where we only discuss the Brownian motion case as the original model of \cite{BB07,BPSV83} did. It was revealed in \cite{IV01} that stable noise is also important in the stochastic resonance climate model.
In this section, we consider the following multiscale time-inhomogeneous SDEs driven by additive $\alpha$-stable noise, which provides a suitable framework for modelling weather-climate dynamics subject to perturbations caused by extreme events. For simplicity, we only consider the case $\alpha=1.5$, i.e.,
\begin{equation}\label{example1}
    \begin{cases}
      \d X_t^\epsilon =\left(6000\left(1-\exp\left(-\frac{50}{(X_t^\epsilon)^2}\right)\right)(X_t^\epsilon-(X_t^\epsilon)^3)+660\cos(2\pi t)+300 Y_t^\epsilon\right)\d t+3.41\d L_t^1,~X_s^\ve=x, \\
      \d Y_t^\epsilon =\frac{1}{\ve}\left(2336+574.5X_t^\epsilon-365\pi Y_t^\epsilon+3796\cos \left(\frac{2\pi t}{\ve}+2.858\right)\right)\d t+\frac{2.55}{\ve^{2/3}}\d L_t^2,~Y_s^\ve=y.
   \end{cases}
\end{equation}
The reason for including the term $1-\exp(-50/x^2)$ in the above can be found in \cite{FLZ21}. Set 
\[b(t,x,y)=6000(1-\exp(-50x^{-2}))(x-x^3)+300 y+660\cos(2\pi t),~~\sigma(t,x)=3.41,\]
and
\[f(t,x,y)=2336+574.5x-365\pi y+3739\cos(2\pi t+2.858),~~g(t,y)=2.55.\]
It is not difficult to see that Assumptions \ref{cond} and \ref{con:periodic} hold with $\tau_1=1$ and $\tau_2=1$. Then the multiscale system \eqref{example1} admits a unique solution $(X_t^\ve,Y_t^\ve)$. Consequently, the corresponding time-inhomogeneous frozen equation is
\begin{equation*}
    \begin{cases}
        \d Y_t=\left(2336+574.5x-365\pi Y_t^\ve+3796\cos\left(2\pi t+2.853\right)\right)\d t+2.55\d L_t^2,\\
        Y_s=y.
    \end{cases}
\end{equation*}
It follows from Theorem \ref{thm:periodic measures} that the time-inhomogeneous equation above has a unique periodic measure $\rho_t^x$. According to Theorem \ref{thm:X_t^ve-bar{X}_t^ve}, for any $p\in(1,\alpha)$, it holds that
\begin{align*}
    \sup_{t\in[s,T]}\be[|X_t^\ve-\bar{X}_t^\ve|^p]\leq C_{\kappa,p,s,T}\ve^\frac{p-1}{\alpha+p-1}.
\end{align*}
Here, $\bar{X}_t^\ve$ satisfies the following averaged equation:
\begin{equation}\label{eq:w-c}
    \begin{cases}
        \d \bar{X}_t^\ve=\bar{b}(t,t/\ve,\bar{X}_t^\ve)\d t+3.41\d L_t^1,\\
        \bar{X}_s^\ve=x,    
    \end{cases}
\end{equation}
where $\be^x[Y_t]=\int_{\br}y\rho_t^x(\d y)$ and 
\[\bar{b}(t_1,t_2,x)=6000(1-\exp(-50x^{-2}))(x-x^3)+300\be^x[Y_{t_2}]+600\cos(2\pi t_1).\]
It is not difficult to see that $\bar{b}(t_1,t_2,x)$ is periodic in $t_1$ and $t_2$ with periods $1$ and $1$. Hence, SDE (\ref{eq:w-c}) is a random quasi-periodic system when the reciprocals of $1$ and $\ve$ are rationally linearly independent. 
Moreover, averaging $\bar{b}(t_1, t_2, x)$ over one period in the variable $t_2$, we obtain
\begin{align*}
    \bar{b}(t_1,x)=\int_0^1(6000(1-\exp(-50x^{-2}))(x-x^3)+300\be^x[Y_{t_2}]+600\cos(2\pi t_1))\d t_2.
\end{align*}
For any $p\in(1,\alpha)$, it follows from Theorem \ref{thm:X_t^ve-bar X} that
\begin{equation*}
    \sup_{t\in[s,T]}\be[|X_t^\ve-\bar{X}_t|^p]\leq C_{\kappa,p,s,T}\ve^\frac{p-1}{\alpha+p-1},
\end{equation*}
where $\bar{X}_t$ is the solution to the corresponding averaged equation:
\begin{equation*}
    \begin{cases}
        \d \bar{X}_t=\bar{b}(t,\bar{X}_t)\d t+3.41\d L_t^1,\\
        \bar{X}_s=x.    
    \end{cases}
\end{equation*}
Furthermore, it is not difficult to see that this is a random periodic system.

\appendices
\section{The climate-weather multiscale system}\label{A}
In this appendix, we introduce the weather and climate systems and couple them together to obtain a multiscale system, which provides a motivating example for our model \eqref{SDEs}.

The climate system proposed in Benzi, Parisi, Sutera, and  Vulpiani \cite{BPSV83} is of the following form:
\begin{equation}\label{eq:BPSV climate modle}
       \d X_t=(\alpha X_t-\beta X_t^3+A\cos(\omega t))\d t+\sigma\d W_t^1,
\end{equation}
where $\alpha=\beta=1$, $A=0.11$, $\omega=\frac{2\pi}{6000}$, $\sigma=0.25$ and $W_t^1$ is a Brownian motion on $\br^d$. This SDE is a random periodic model with a rescaled period $6000$. 
We would like to put the climate dynamics along with the whether system in a single multiscale system, where we take the latter as a system with period 1 (year). So we need to further rescale the time in SDE \eqref{eq:BPSV climate modle} accordingly. For this, 
letting $t=0.06u$, still using $t$ as the new time variable instead of $u$, we have
\begin{equation}\label{eq:climate system}
        \d \tilde{X}_t=0.06(\tilde{X}_t-\tilde{X}_t^3+0.11\cos(0.00002\pi t))\d t+0.25\sqrt{0.06}\d \tilde{W}_t.
\end{equation}
This makes the period of the system to be as $100000$, which is the approximate period of climate change with a time unit in ``years''. Let $\ve=10^{-5}$, $\alpha=\beta=6000\ve$, $\omega=2\pi\ve$,  $A=660\ve$, and $\sigma=0.25\sqrt{6000\ve}\approx6.12\sqrt{\ve}$ in \eqref{eq:BPSV climate modle}, then \eqref{eq:climate system} can be rewritten as
\begin{equation}\label{eq: general climat system}
        \d \tilde{X}_t= (6000\ve \tilde{X}_t-6000\ve \tilde{X}_t^3+660\ve\cos (2\pi \ve t))\d t+6.12\sqrt{\ve}\d \tilde{W}_t^1,.
\end{equation}
This is a special case of the general SDE:
\begin{equation}
    \d \tilde{X}_t=\ve b(\ve t,\tilde{X}_t)\d t+\sqrt{\ve}\sigma(\ve t,\tilde{X}_t)\d \tilde{W}_t,
\end{equation}
where
\[b(t,x)=6000x-6000x^3+660\cos(2\pi t),~~\sigma(t,x)=6.12.\]

The weather system proposed by F. Benth and J. Benth \cite{BB07} is given the following SDE:
\begin{equation*}
        \d Y_t=(c_0+c_1\cos(2\pi(t-c_2)/365)-\pi Y_t)\d t+g\d W_t^2,
\end{equation*}
where $c_0=6.4, c_1=10.4, c_2=-166, g=0.3$ and $W_t^2$ is a Brownian motion on $\br^d$. The time unit in this equation is the ``day''.
Let $t=365 u$ make the time unit ``year'' and still keep $t$ rather than $u$ as the time variable to have
\begin{equation}\label{eq:weathe system}
        \d\tilde{Y}_t=f(t,\tilde{Y}_t)\d t+g(t,\tilde{Y}_t)\d \tilde{W}_t^2,
\end{equation}
where
\[f(t,y)=2336+3796\cos(2\pi(t+0.4548))-365\pi y,~~g(t,y)=5.73.\]

Under the same time unit for both SDEs \eqref{eq: general climat system} and \eqref{eq:weathe system}, we consider the following fully coupled multiscale system:
\begin{equation}\label{eq:coupled system}
    \begin{cases}
        \d \tilde{X}_t^\ve=\ve b(\ve t,\tilde{X}_t^\ve,\tilde{Y}_t^\ve)\d t+\sqrt{\ve}\sigma(\ve t,\tilde{X}_t^\ve)\d\tilde{W}_t^1,\\
        \d\tilde{Y}_t^\ve=f(t,\tilde{X}_t,\tilde{Y}_t^\ve)\d t+g(t,\tilde{Y}_t)\d\tilde{W}_t^2,
    \end{cases}
\end{equation}
where
\[b(t,x,y)=6000(x-x^3)+660\cos(2\pi t)+300 y,~~\sigma(t,x)=6.12,\]
and
\[f(t,x,y)=2336+574.5x-365\pi y+3796\cos(2\pi(t+0.4548)),~~g(t,y)=5.73.\]
In the averaging theory of multiscale systems, the following SDE with time rescaled from \eqref{eq:coupled system} is normally considered,
\begin{equation*}
    \begin{cases}
        \d \ddot{X}_t^\ve=b(t,\ddot{X}_t^\ve,\ddot{Y}_t^\ve)\d t+\sigma(t,\ddot{X}_t^\ve)\d\ddot{W}_t^1,\\
        \d\ddot{Y}_t^\ve=\frac{1}{\ve}f(t/\ve,\ddot{X}_t^\ve, \ddot{Y}_t^\ve)\d t+\frac{1}{\sqrt{\ve}}g(t/\ve,\ddot{Y}_t)\d \ddot{W}_t^2.
    \end{cases}
\end{equation*}
This is exactly the SDE (\ref{SDEs}) with periods $1$ and $\epsilon$, but with Brownian motions as the driving noise instead of the stable processes considered in this paper. The drift term $b$ in this case is not Lipschitz, but can be modified to a Lipschitz function as suggested in \cite{FLZ21}, see \eqref{example1} for details.

\section*{Acknowledgements}
We acknowledge the financial supports from an EPSRC grant (ref EP/S005293/2), and a Royal Society Newton Fund grant (ref. NIF\textbackslash R1\textbackslash 221003).

\addtolength{\itemsep}{-1.5 em} 
\setlength{\itemsep}{-3pt}
\footnotesize

\addcontentsline{toc}{section}{References}

\bibliographystyle{plain}  
\bibliography{reference}   

\end{document}